\documentclass[12pt,reqno]{amsart}

\newtheorem{theorem}{Theorem}[section]
\newtheorem{lemma}[theorem]{Lemma}

\theoremstyle{definition}   
\newtheorem{definition}{Definition}
\newtheorem{example}[theorem]{Example}

\theoremstyle{remark}

\newtheorem{conjecture}[theorem]{Conjecture}

\numberwithin{equation}{section}

\usepackage{times}
\usepackage{enumerate}
\usepackage{graphicx,adjustbox}
\usepackage{hyperref}
\usepackage{amsmath,amssymb}

\usepackage{amscd}
\usepackage{graphicx}
\usepackage[all]{xy}

\title[Numerical Semigroups Generated by Concatenation]
{Numerical Semigroups Generated by Concatenation of Arithmetic Sequences}
\author{
Ranjana Mehta
\and
Joydip Saha
\and
Indranath Sengupta
}
\date{}

\address{\small \rm  Discipline of Mathematics, IIT Gandhinagar, Palaj, Gandhinagar, 
Gujarat 382355, INDIA.}
\email{mehta.n.ranjana@gmail.com}

\address{\small \rm  Discipline of Mathematics, IIT Gandhinagar, Palaj, Gandhinagar, 
Gujarat 382355, INDIA.} 
\email{saha.joydip56@gmail.com}
\thanks{The second author thanks SERB, Government of India for the Research Associate 
position at IIT Gandhinagar, through the research project EMR/2015/000776.}

\address{\small \rm  Discipline of Mathematics, IIT Gandhinagar, Palaj, Gandhinagar, 
Gujarat 382355, INDIA.}
\email{indranathsg@iitgn.ac.in}
\thanks{The third author is the corresponding author. The author 
thanks SERB, Government of India for their support through the project EMR/2015/000776.}

\date{}



\allowdisplaybreaks

\begin{document}

\begin{abstract}
We introduce the notion of numerical semigroups generated by concatenation 
of arithmetic sequences and show that this class of numerical semigroups 
exhibit multiple interesting behaviours. 
\end{abstract}
\maketitle

\section{introduction}
A \textit{numerical semigroup} $\Gamma$ is a subset of the set of nonnegative integers 
$\mathbb{N}$, closed under addition, contains zero and generates $\mathbb{Z}$ as 
a group. It follows that (see \cite{rgs}) the set $\mathbb{N}\setminus \Gamma$ is 
finite and that the semigroup $\Gamma$ has a unique minimal system of generators 
$n_{0} < n_{1} < \cdots < n_{p}$. The greatest integer not belonging to $\Gamma$ 
is called the \textit{Frobenius number} of $\Gamma$, denoted by $F(\Gamma)$. The integers 
$n_{0}$ and $p + 1$ are known as the \textit{multiplicity} and the 
\textit{embedding dimension} of the semigroup $\Gamma$, usually 
denoted by $m(\Gamma)$ and $e(\Gamma)$ respectively. The 
\textit{Ap\'{e}ry set} of $\Gamma$ with respect to a non-zero $\mathbf{a}\in \Gamma$ is 
defined to be the set $\rm{Ap}(\Gamma,\mathbf{a})=\{s\in \Gamma\mid s-\mathbf{a}\notin \Gamma\}$. 
The numerical semigroup $\Gamma$ is \textit{symmetric} if $F(\Gamma)$ is odd and 
$x\in \mathbb{Z}\setminus \Gamma$ implies $F(\Gamma)-x\in \Gamma$. Given 
integers $n_{0} < n_{1} < \cdots < n_{p}$; the map 
$\nu : k[x_{0}, \ldots, x_{p}]\longrightarrow k[t]$ defined as 
$\nu(x_{i}) = t^{n_{i}}$, $0\leq i\leq p$ defines a parametrization 
for an affine monomial curve; the ideal $\ker(\nu)=\mathfrak{p}$ is called the 
defining ideal of the monomial curve defined by the parametrization 
$\nu(x_{i}) = t^{n_{i}}$, $0\leq i\leq p$. Kunz \cite{kunz} proved that the semigroup 
$\Gamma$ generated by the integers $n_{0}, \ldots , n_{p}$ is symmetric if 
and only if the monomial curved parametrized by 
$x_{0}=t^{n_{0}}, \ldots , x_{p}=t^{n_{p}}$ is Gorenstein. The defining ideal 
$\mathfrak{p}$ is a graded ideal with respect to the weighted gradation 
and therefore any two minimal generating sets of $\mathfrak{p}$ have the same cardinality. 
Similarly, by an abuse of notation, one can define a semigroup homomorphism 
$\nu: \mathbb{N}^{p+1} \rightarrow \mathbb{N}$ as 
$\nu((a_{0}, \ldots , a_{p})) = a_{0}n_{0} + a_{1}n_{1} + \cdots + a_{p}n_{p}$. 
Let $\sigma$ denote the kernel of congruence of the map $\nu$. It is known 
that $\sigma$ is finitely generated. The minimal number of generators of the 
ideal $\mathfrak{p}$, that is, the cardinality of a minimal generating set of 
$\mathfrak{p}$ is the same as the minimal cardinality of a system of generators 
of $\sigma$. Henceforth, we will be using these two notions interchangeably.
\medskip

The central 
theme behind this work is to understand the situations, when, given an integer 
$e$, the minimal number of generators for the defining ideal $\mathfrak{p}$ 
is a bounded function of the embedding dimension $e(\Gamma)=e$. The symmetry 
condition on $\Gamma$ is a good condition to start with. Given an integer 
$e\geq 4$, it is not completely understood whether the symmetry condition 
on $\Gamma$ in embedding dimension $e$ ensures that the minimal cardinality of 
a system of generators of $\sigma$ defined above is a bounded function of $e$. 
This was answered in affirmative by Bresinsky for $e=4$ in \cite{bre1}, and 
for certain cases of $e=5$ in \cite{bre2}. Rosales \cite{ros1} 
constructed numerical semigroups for a given multiplicity $m$ and 
embedding dimension $e$, which are symmetric, and showed that the 
cardinality of a minimal presentation of these semigroups is a 
bounded function of the embedding dimension $e$. In fact using 
the pertinent results obtained in \cite{bre1}, \cite{herzog}, 
\cite{ros3}, \cite{rgs1}, one can compute the cardinality of a 
minimal presentation of a symmetric numerical semigroup with 
multiplicity $m\leq 8$. This remains an open question in general, 
whether symmetry condition on the numerical semigroup $e\geq 5$ 
imposes an upper bound on the cardinality of a minimal presentation 
of a numerical semigroup $\Gamma$. On the other hand, Bresinsky \cite{bre} 
produced a class of examples in embedding dimension $4$ such that 
the cardinality of a minimal presentation is unbounded. It was proved 
in \cite{mss2} that all the Betti numbers of Bresinsky's example of 
curves are unbounded. It is unknown if one can define a non-degenerate 
class of monomial curves in the affine space $\mathbb{A}^{e}$ such that the defining ideal 
requires an unbounded number of generators. Our work in this paper 
has helped us converge to the Conjecture \ref{conjecture} 
in section 3 on the unboundedness of minimal number of generators of the 
defining ideal in arbitrary embedding dimension. 
\medskip

The main aim of the paper 
is to show that by one common method of \textit{concatenation of arithmetic sequences}, 
we can understand at least three different classes 
of numerical semigroups which exhibit diverse characteristics like 
unboundedness of the number of generators of a minimal presentation (in 
section 3), symmetry (in section 4) and boundedness of the number of generators 
of a minimal presentation (in section 5). These three constructions 
have been named as the \textit{unbounded concatenation}, the \textit{symmetric concatenation} 
and the \textit{almost maximal concatenation} respectively. The symmetric concatenation 
also has the expected boundedness of the number of minimal relations. Let us 
introduce the notion of concatenation first before we move on to the specific classes 
in sections 3, 4 and 5.
\bigskip

\section{Concatenation of arithmetic sequences} 
Let us first recall Bresinsky's examples of monomial curves in $\mathbb{A}^{4}$, 
as defined in \cite{bre}. Let $q_2\geq 4$ be even; 
$q_{1} = q_{2}+1,\, d_1=q_{2}-1$. Set $n_{1}=q_{1}q_{2},\, n_{2}=q_{1}d_{1},\, 
n_{3}=q_{1}q_{2}+d_{1},\, n_{4}=q_{2}d_{1}$. It is clear that 
$\gcd (n_1,\, n_2,\, n_3,\, n_4)=1$. Bresinsky's 
examples lead us to the notion of concatenation of arithmetic sequences, 
which we define below.
\medskip

Let $e\geq 4$. Consider the sequence of positive integers 
$a<a+d<a+2d<\ldots<a+(n-1)d<b<b+d<\ldots<b+(m-1)d,$ where $m,n\in \mathbb{N}$, $m+n=e$  
and $\gcd(a,d)=1$. Let us assume that this sequence minimally generates 
the numerical semigroup $\Gamma=\langle a, a+d, a+2d,\ldots, a+(n-1)d, b, b+d,\ldots, b+(m-1)d\rangle$. 
Then, $\Gamma$ is called the \textit{numerical semigroup generated 
by concatenation of two arithmetic sequences with the same common difference $d$}. 
A comment is in order. The definition of 
concatenation includes the minimality of the sequence of integers generating the 
numerical semigroup $\Gamma$. For example, if we take $e=d=4$, $a=5$ and $b=10$, then 
the concatenated sequence that we get is $5, 9, 10, 14$, which is clearly not 
minimal. We will not consider such sequences for our purpose in order to 
avoid degeneracy. In fact, we will verify the minimality of the sequence first (see \ref{minimal1}, \ref{definegamma}, \ref{minimalcondition}, \ref{minimalexample}). 
\medskip

It can be easily verified that, up to a renaming of integers, Bresinsky's example 
defined above fall under this category. In this article, we show that the construction 
of concatenation shows diverse behaviours. The most interesting behaviour is exhibited 
by the construction in section. Here we produce the concatenated class in 
embedding dimension $4$, that requires unbounded number of elements in a minimal 
presentation. This class of numerical semigroups will be referred to as the 
\textit{unbounded concatenation}. The case for arbitrary embedding dimension 
has been posed as a conjecture. Moreover, it is worthwhile to calculate the 
pseudo-Frobenius numbers and type in this case, which we anticipate would 
exhibit similar unbounded behaviour. The reason behind this anticipation is the 
following: It was proved in \cite{mss2} that the Betti numbers of 
Bresinsky's examples are unbounded. It is interesting to note that the 
Cohen-Macaulay type of a numerical semigroup ring $k[\Gamma]$ is given 
by the last Betti number and that is given by the number $\mu(Der_{k}(k[\Gamma]) + 1$ 
(see Corollary 6.2 in \cite{patsen}), where $\mu(Der_{k}(k[\Gamma])$ denotes 
the minimal number of generators of the derivation module $Der_{k}(k[\Gamma])$. 
\medskip

In section 4, we produce a concatenated 
family which are symmetric numerical semigroups; we name this as the 
\textit{symmetric concatenation} and this generalizes the construction by 
Rosales in \cite{ros1}. We calculate the Ap\'{e}ry set in order to prove that the 
cardinality of minimal presentation is a bounded function of the embedding 
dimension $e$ and in turn give an affirmative answer to the question on the 
boundedness of the number of minimal relations for a symmetric numerical 
semigroup. Moreover, the Frobenius number and the pseudo-Frobenius numbers are 
the same in this case and the type will be $1$.
\medskip

Numerical semigroups with the property ``multiplicity= embedding dimension+1" has 
been studied before in \cite{rgs1}, where it was proved that the minimal number 
of generators for the defining ideal of this class of numerical semigroups 
is a bounded function of $e$. In section 5 we examine numerical semigroups 
formed by concatenation together with the condition ``multiplicity= embedding 
dimension+1"; we call this the \textit{almost maximal concatenation}. We calculate 
the Ap\'{e}ry set, the Frobenius number and the pseudo-Frobenius numbers explicitly. 
We also give a complete description of a minimal generating set for the ideal 
$\mathfrak{p}$ for this class of numerical semigroups.   

\section{Unbounded concatenation}
We show that the concatenation construction is likely to give us examples of 
numerical semigroups in arbitrary embedding dimension with unbounded minimal 
presentation. Let us first propose our example in 
arbitrary embedding dimension and then focus on embedding dimension $4$. The case 
for an arbitrary embedding dimension has been posed as a conjecture. 
\medskip

\begin{lemma} Let $e\geq 4$, $n\geq 5$ and $q\geq 0$. Let us define
$m_{i}:=n^2+(e-2)n+q+i$, for $0\leq i\leq e-3 $ and $m_{e-2}:=n^2+(e-1)n+q+(e-3)$, 
$m_{e-1}:=n^2+(e-1)n+q+(e-2)$. Let 
$\mathfrak{S}_{(n,e,q)}=\langle m_{0},\ldots, m_{e-1}\rangle$, then 
$\{m_{0},\ldots,m_{e-1}\}$ is a minimal generating set for the semigroup 
$\mathfrak{S}_{(n,e,q)}$.
\end{lemma}

\proof It is easy to observe that $m_{i}+m_{j}>m_{k}$ for $0\leq i,j,k\leq e-1$. 
Therefore $\{m_{0},\ldots,m_{e-1}\}$ is a minimal generating set for the semigroup 
$\mathfrak{S}_{(n,e,q)}$. \qed
\medskip

Let $e\geq 4$, $n\geq 5$, $q\geq 0$ and  $\mathcal{Q}_{(n,e,q)}\subset k[x_{0},\ldots,x_{e-1}]$ 
be defining ideal of $\mathfrak{S}_{(n,e,q)}$.
\medskip

\begin{lemma}\label{minimal1}
Let  $\mathfrak{h}_{i}=x_{0}^{i}x_{1}^{n+2-i}-x_{e-2}^{i}x_{e-1}^{n+1-i}$ for 
$0\leq i\leq n+1 $. Then $\{\mathfrak{h}_{0},\ldots,\mathfrak{h}_{n+1}\}$ is 
minimal and contained in the ideal $\mathcal{Q}_{(n,e,e-4)}\subset k[x_{0},\ldots,x_{e-1}]$.
\end{lemma}

\proof Since $im_{0}+(n+2-i)m_{1}=im_{e-2}+(n+1-i)m_{e-1}$, we have 
$\{\mathfrak{h}_{0},\ldots, \mathfrak{h}_{n+1}\}\subset \mathcal{Q}_{(n,e,e-4)}$. 
Suppose $\mathfrak{h}_{i}=\displaystyle \sum_{j=0,j\neq i}^{n+1}g_{ij}\mathfrak{h}_{j}$, 
where $g_{ij}\in k[x_{0},\ldots,x_{e-1}] $ for $0\leq i\leq n+1$. If we take 
$x_{e-2}=x_{e-1}=0$, then we get 
$$x_{0}^{i}x_{1}^{n+2-i}= \displaystyle \sum_{j=0,j\neq i}^{n+1}g^{'}_{ij}x_{0}^{j}x_{1}^{n+2-j}=\displaystyle \sum_{j=0}^{i-1}g^{'}_{ij}x_{0}^{j}x_{1}^{n+2-j}+\displaystyle \sum_{j=i+1}^{n+1}g^{'}_{ij}x_{0}^{j}x_{1}^{n+2-j}=\mathfrak{A}+\mathfrak{B},$$  
where $g^{'}_{ij}\in k[x_{0},\ldots,x_{e-3}] $ for $0\leq i\leq n+1$. Since power of 
$x_{1}$ in each monomial in $\mathfrak{A}$ is greater than $n+2-i $ and power of 
$x_{0}$ in each monomial in $\mathfrak{B}$ is greater than $i$, we get a contradiction.\qed
\bigskip

\begin{conjecture}\label{conjecture}
The set $\{\mathfrak{h}_{0},\ldots,\mathfrak{h}_{n+1}\}$ is a part of a minimal 
generating set for the ideal  $\mathcal{Q}_{(n,e,e-4)}\subset k[x_{0},\ldots,x_{e-1}]$, 
hence $\mu(\mathcal{Q}_{(n,e,e-4)})\geq n+2$. Moreover, the set 
$\{\mu(\mathcal{Q}_{(n,e,q)})\mid n\geq 5,e\geq 4, q\geq 0\}$ is unbounded above. 
\end{conjecture}
\medskip

Let us now consider the above of class of numerical semigroups for the embedding 
dimension $e=4$ and show that they indeed have an unbounded minimal presentation. 
Let us first define the following special sets of binomials.
$$A_{1}=\{f_{\mu}\mid f_{\mu}=x_{2}^{(n+1)-\mu}x_{3}^{\mu}-x_{0}^{(n+1)-\mu}x_{1}^{\mu+1}, 0\leq \mu \leq n+1\},$$
$$A_{2}=\{h_{t}\mid h_{t}=x_{1}^{(n+1)-t}x_{3}^{t}-x_{0}^{n-t}x_{2}^{t+1}, 0\leq t\leq n\},$$
$$g_{1}=x_{0}^{n+1}-x_{3}^{n}, \quad g_{2}=x_{1}x_{2}-x_{0}x_{3}.$$

\begin{theorem} The set $J:=A_{1}\cup A_{2}\cup\{g_{1},g_{2}\}$ is a generating set 
for the ideal $\mathcal{Q}_{(n,4,0)}\subset k[x_{0},x_{1},x_{2},x_{3}]$.
\end{theorem}

\begin{proof} We know that 
$$\mathfrak{S}_{(n,4,0)}=\langle n^{2}+2n, n^{2}+2n+1, n^{2}+3n+1, n^{2}+3n+2\rangle.$$ 
We use Theorem 4.8 in \cite{eto} in order to prove that $J:=A_{1}\cup A_{2}\cup\{g_{1},g_{2}\}$ 
is a generating set for the ideal $\mathcal{Q}_{(n,4,0)}\subset k[x_{0},x_{1},x_{2},x_{3}]$. 
It is therefore enough to prove the following claim:
\medskip
 
\noindent\textbf{Claim.} $\dim_{k} (R/(J+\langle x_{0}\rangle)) = n^{2}+2n$.
\medskip

\noindent\textbf{Proof of the Claim.} We have 
 $$J+\langle x_{0}\rangle =\langle \{x_{1}^{n+1},x_{2}^{n+1}, x_{2}^{n} x_{3},\ldots, x_{2}^{2}x_{3}^{n-1}, x_{1}^{n}x_{3},\ldots x_{1}^{2}x_{3}^{n-1}, x_{3}^{n},x_{1}x_{2},x_{0}\}\rangle.$$
Let us define 
$$J'=\langle \{x_{1}^{n+1},x_{2}^{(n+1)}, x_{2}^{n} x_{3},\ldots, x_{2}^{2}x_{3}^{n-1}, x_{1}^{n}x_{3},\ldots x_{1}^{2}x_{3}^{n-1}, x_{3}^{n},x_{1}x_{2}\}\rangle.$$
Therefore, $k[x_{0},x_{1},x_{2},x_{3}]/(J+\langle x_{0}\rangle)= k[x_{1},x_{2},x_{3}]/J'$. It is 
evident that the $k$-vector space $k[x_{1},x_{2},x_{3}]/J'$ is spanned by the set 
\begin{align*}
&\{1\}\cup \{x_{1}^{p}\mid 1\leq k\leq n\}\cup \{x_{2}^{l}\mid 1\leq l\leq n\} \cup \{x_{3}^{m}\mid 1\leq m \leq n-1 \}\\
&\cup\{
x_{1}^{r}x_{3}^{s}\mid 1\leq s\leq n-1, 1\leq r\leq n-s\}\cup \{
x_{1}^{r'}x_{3}^{s'}\mid 1\leq s'\leq n-1, 1\leq r'\leq n-s\}. 
\end{align*}

\noindent The cardinality of the above set is $n^{2}+2n$, which proves the claim.  
\end{proof}
\medskip

Note that $f_{n+1}=-(x_{3}\cdot g_{1}+x_{0}^{n}\cdot g_{2}+x_{1}\cdot h_{0})$, 
$f_{n}=-(x_{2}\cdot g_{1}+x_{0}\cdot h_{0})$ and $f_{0}=-(x_{1}\cdot g_{1}+h_{n})$. 
Let us define 
$$A'_{1}:=\{f_{\mu}\mid f_{\mu}=x_{2}^{(n+1)-\mu}x_{3}^{\mu}-x_{0}^{(n+1)-\mu}x_{1}^{\mu+1}, 1\leq \mu \leq n-1\}.$$
\medskip

\begin{theorem}
The set $A'_{1}\cup A_{2}\cup\{g_{1},g_{2}\}$ is a minimal generating set for the ideal 
$\mathcal{Q}_{(n,4,0)}\subset k[x_{0},x_{1},x_{2},x_{3}]$. Hence 
$\mu(\mathcal{Q}_{(n,4,0)})= 2(n+1)$. 
\end{theorem}

\begin{proof} Let $\pi_{01}:k[x_{0},x_{1},x_{2},x_{3}]\rightarrow k[x_{2},x_{3}]$ be 
such that $\pi_{01}(x_{0})=0$, $\pi_{01}(x_{1})=0$, $\pi_{01}(x_{2})=x_{2}$ and 
$\pi_{01}(x_{3})=x_{3}$. Given $1\leq \mu'\leq n-1$, let 
$$f_{\mu'}=x_{2}^{(n+1)-\mu'} x_{3}^{\mu'}-x_{0}^{(n+1)-\mu'} x_{1}^{\mu'+1} = 
\sum_{\mu=1, \mu \neq \mu'}^{n-1} \alpha_{\mu}f_{\mu}+\sum_{t=0}^{n} \beta_{t}h_{t}+\gamma_{1}g_{1}+\gamma_{2}g_{2},$$ 
for $\alpha_{\mu}, \beta_{t}, \gamma_{1}, \gamma_{2}\in k[x_{0},x_{1},x_{2},x_{3}]$. 
Applying $\pi_{01}$ on both sides of the above equation we get 
$$x_{2}^{(n+1)-\mu'} x_{3}^{\mu'}= \sum_{\mu=1, \mu\neq \mu'}^{n-1}\alpha_{\mu}(0,0,x_{2},x_{3})\left(x_{2}^{(n+1)-\mu}x_{3}^{\mu}\right)-\gamma_{1}(0,0,x_{2},x_{3})x_{3}^{n}.$$
We know that $1\leq \mu'\leq n-1$ and $\mu = \mu'$, therefore in the above expression 
the monomial $x_{2}^{(n+1)-\mu'} x_{3}^{\mu'}$ can not appear in the terms of 
$\gamma_{1}(0,0,x_{2},x_{3})x_{3}^{n}$.  If $\mu '< \mu $, then in each monomial 
appearing on the right hand side of the above expression, the indeterminate $x_{3}$ 
appears with an exponent that is strictly greater than $\mu'$; hence the expression 
is absurd. If $\mu '> \mu$, then $(n+1)-\mu '<(n+1)-\mu$. Therefore, in each monomial 
appearing on the right hand side of the above expression the indeterminate $x_{2}$ 
appears with an exponent that is strictly greater than $(n+1)-\mu'$; again absurd. 
Hence the above expression is not possible.
\medskip

Let $\pi_{13}:k[x_{0},x_{1},x_{2},x_{3}]\rightarrow k[x_{0},x_{2}]$ be such that 
$\pi_{13}(x_{0})=x_{0}$, $\pi_{13}(x_{1})=0$, $\pi_{13}(x_{2})=x_{2}$ and 
$\pi_{13}(x_{3})=0$. Let $h_{t'}=x_{1}^{(n+1)-t'}x_{3}^{t'}-x_{0}^{n-t'}x_{2}^{t'+1}$ 
for a given $0\leq t'\leq n$. Suppose that 
$$h_{t'}=\sum_{\mu=1}^{n-1} \alpha_{\mu}f_{\mu}+\sum_{t=0, t\neq t'}^{n} \beta_{t}h_{t}+\gamma_{1}g_{1}+\gamma_{2}g_{2},$$ 
for some $\alpha_{\mu}, \beta_{t}, \gamma_{1}, \gamma_{2}\in k[x_{0},x_{1},x_{2},x_{3}]$. 
Applying $\pi_{13}$ on both sides of the above expression we get 
$$x_{0}^{n-t'}x_{2}^{t'+1}=\sum_{t=0,t\neq t'}^{n}\beta_{t}(x_{0},0,x_{2},0)x_{0}^{n-t}x_{2}^{t+1}+\gamma_{1}(x_{0},0,x_{2},0)x_{0}^{n+1}.$$
We have $0\leq t'\leq n$, therefore the monomial $x_{0}^{n-t'}x_{2}^{t'+1}$ can not 
be written in the terms of $\gamma_{1}(0,0,x_{2},x_{3})x_{0}^{n+1}$.  If $t '< t$ 
then each monomial appearing on the right hand side of the above equation has exponent 
of $x_{2}$ strictly greater than $t'+1$. If $t '> t $ then $n-t'< n-t$, therefore 
each monomial appearing on the right hand side of the above equation has exponent 
of $x_{0}$ strictly greater than $n-t'$. Hence the above expression is impossible.
\medskip

Suppose that  
\begin{align*}
g_{1}=x_{0}^{n+1}-x_{3}^{n}&=\sum_{\mu=1}^{n-1}\alpha_{\mu}\left[x_{2}^{(n+1)-\mu}x_{3}^{\mu}-x_{0}^{(n+1)-\mu}x_{1}^{\mu+1}\right]\\
&+\sum_{t=0}^{n}\beta_{t}\left[x_{1}^{(n+1)-t}x_{3}^{t}-x_{0}^{n-t}x_{2}^{t+1}\right]
+\gamma_{2}(x_{1}x_{2}-x_{0}x_{3}),
\end{align*}
where $\alpha_{\mu}, \beta_{t}, \gamma_{1}, \gamma_{2}\in k[x_{0},x_{1},x_{2},x_{3}]$. 
Let $\pi_{12}:k[x_{0},x_{1},x_{2},x_{3}]\rightarrow k[x_{0},x_{3}]$ be such that 
$\pi_{12}(x_{0})=x_{0}, \pi_{12}(x_{1})=0, \pi_{12}(x_{2})=0$ and $\pi_{12}(x_{3})=x_{3}$. 
Applying $\pi_{12}$ on both sides of the above equation we get 
$$x_{0}^{n+1}-x_{3}^{n}=-\gamma_{2}(x_{0},0,0,x_{3})x_{0}x_{3},$$ 
which is absurd. 
\medskip

Suppose that 
\begin{align*}
g_{2}=x_{1}x_{2}-x_{0}x_{3}&=\sum_{\mu=1}^{n-1}\alpha_{\mu}\left[x_{2}^{(n+1)-\mu}x_{3}^{\mu}-x_{0}^{(n+1)-\mu}x_{1}^{\mu+1}\right]\\
&+\sum_{t=0}^{n}\beta_{t}\left[x_{1}^{(n+1)-t}x_{3}^{t}-x_{0}^{n-t}x_{2}^{t+1}\right]
+\gamma_{1}(x_{0}^{n+1}-x_{3}^{n}),
\end{align*}
where $\alpha_{\mu}, \beta_{t}, \gamma_{1}, \gamma_{2}\in k[x_{0},x_{1},x_{2},x_{3}]$. 
Applying $\pi_{12}$ on both sides of the equation we get 
$$x_{0}x_{3}=-\gamma_{1}(x_{0},0,0,x_{3})(x_{0}^{n+1}-x_{3}^{n}),$$ 
which is absurd.
\end{proof}

\section{Symmetric concatenation}
In this section we construct a class of symmetric numerical semigroup of embedding dimension $e \geq 4,$ which is a generalisation of Rosales' result, in the sense that if we consider the case $d=1$ it gives Rosales' construction given in \cite{ros1}. We prove that the cardinality of a minimal presentation of the semigroup is a bounded function of the embedding dimension $e$.

\subsection{Numerical semigroups $\Gamma_{(e,q,d)}(\mathcal{S})$ 
and $\Gamma_{(e,q,d)}(\mathcal{T})$}

\begin{theorem}\label{definegamma}
\begin{enumerate}
\item Let $e\geq 4$ be an integer, $q$ a positive integer and $m=e+2q+1$. 
Let $d$ be a positive integer that satisfies $\gcd(m,d)=1$. Let us define 
$\mathcal{S}=\{m,m+d,(q+1)m+(q+2)d,(q+1)m+(q+3)d,\ldots,(q+1)m+(q+e-1)d \}$. The set $\mathcal{S}$ 
is a minimal generating set for the numerical semigroup $\Gamma_{(e,q,d)}(\mathcal{S})$ generated 
by $\mathcal{S}$.
\item Let $e\geq 4$ be an integer, $q$ a positive even integer with 
$q\geq e-4$ and $m=e+2q$. Let $d$ be an odd positive 
integer that satisfies $\gcd(m,d)=1$. Let us define 
$\mathcal{T}=\{m,m+d,q(m+1)+(q-\frac{e-4}{2})d+\frac{e}{2},q(m+1)+(q-\frac{e-4}{2}+1)d+\frac{e}{2},
\ldots,q(m+1)+(q-\frac{e-4}{2}+(e-3))d+\frac{e}{2}\}$. The set $\mathcal{T}$ is a minimal generating 
set for the numerical semigroup $\Gamma_{(e,q,d)}(\mathcal{T})$ generated by $\mathcal{T}$.
\end{enumerate}
\end{theorem}

\proof First of all it is easy to see that both the semigroups $\Gamma_{(e,q,d)}(\mathcal{S})$ 
and $\Gamma_{(e,q,d)}(\mathcal{T})$ are numerical semigroups and that follows from the simple 
observation that $\gcd(m, m+d)=1$. We now prove that both $\mathcal{S}$ and $\mathcal{T}$ are minimal.
\medskip

\noindent (1) Suppose that $(q+1)m+(q+2)d=x_{1}m+x_{2}(m+d)$, where $x_{1},x_{2}\geq 0$ 
are integers. Since $x_{1},x_{2}$ both are positive we have $x_{2} < q+2$. The equation 
$(x_{1}+x_{2}-(q+1))m=(q+2-x_{2})d$ and the fact that $\gcd(m,d)=1$ implies that 
$x_{2}=q+2-lm$, for some integer $l\geq 0$. If $l>0$ then $x_{2}<0$ gives a contradiction. 
If $l=0$ then $x_{2}=q+2$ also contradicts the fact that $x_{2} < q+2$. Therefore, 
$(q+1)m+(q+2)d$ does not belong to the semigroup generated by $m$ and $m+d$. 
\medskip

Similarly, assume that
$$(q+1)m+(q+k)d=x_{1}m+x_{2}(m+d)+\sum_{i=2}^{k-1} t_{i}((q+1)m+(q+i)d),$$
where $x_{1}, x_{2}, t_{i}$ are nonnegative integers. Then, we can write 
$x_{2}=(q+k)-lm-(\sum_{i=2}^{k-1} t_{i}(q+i))$ for some $l\geq 0$. If $l>0$, 
then $x_{2}<0$ gives a contradiction. If $l=0$, then $x_{2}=(q+k)-(\sum_{i=2}^{k-1} t_{i}(q+i))$ 
and therefore $x_{1}=1-k<0$, which also gives a contradiction. 
\medskip

\noindent (2) The proof is similar as in (i). \qed
\medskip

\begin{theorem}\label{apery}
\begin{enumerate}[(1)]
\item The Ap\'{e}ry set $\textrm{Ap}(\Gamma_{(e,q,d)}(\mathcal{S}),m)$ for the numerical 
semigroup $\Gamma_{(e,q,d)}(\mathcal{S})$ with respect to the element $m$ is 
$\beta_{1}\cup\beta_{2}\cup\beta_{3}$, where
\begin{align*}
\beta_{1}& =\{k(m+d)\mid 0\leq k\leq q+1\},\\
\beta_{2}& =\{k(m+d)+(q+1)m+(q+e-1)d \mid 0\leq k\leq q+1\},\\
\beta_{3}& =\{(q+1)m+(q+i)d\mid 2\leq i\leq e-2\}.
\end{align*} 
The Frobenius number of $\Gamma_{(e,q,d)}(\mathcal{S})$ is $4q^2+(2e+2d+4)q+e(d+1)+1$.

\item The Ap\'{e}ry set $\textrm{Ap}(\Gamma_{(e,q,d)},m)(\mathcal{T})$ for the numerical 
semigroup $\Gamma_{(e,q,d)}(\mathcal{T})$ with respect to the element $m$ is $\gamma_{1}\cup \gamma_{2}$, where
\begin{align*}
\gamma_{1} &=\{q(m+1)+(q-\frac{e-4}{2}+k)d+\frac{e}{2}\mid 0\leq k\leq (e-3)\}\\
\gamma_{2} &=\{k(m+d)\mid 0\leq k\leq 2q+1\}.
\end{align*}
\noindent The Frobenius number of $\Gamma_{(e,q,d)}(\mathcal{T})$ is $(e+2q+d)2q+d$.
\end{enumerate}
\end{theorem}

\proof (1) Since $\gcd(m,d)=1$, it is easy to see that the 
elements of $\beta_{1}\cup\beta_{2}\cup\beta_{3}$ form a 
complete residue system modulo $m$. Again $\mathcal{S}$ is minimal 
generating set for $\Gamma_{(e,q,d)}(\mathcal{S})$, therefore elements 
of $\mathcal{S}$ occur in the Ap\'{e}ry set $\textrm{Ap}(\Gamma_{(e,q,d)}(\mathcal{S}),m)$. 
Now we show that $\beta_{1}\subset \textrm{Ap}(\Gamma_{(e,q,d)}(\mathcal{S}),m)$. 
Note that $k(m+d)<(q+1)m+(q+2)d$, for each $0\leq k\leq q+1$. Suppose 
that $x_{1}(m+d)+\sum\limits_{i=2}^{e-1}t_{i}((q+1)m+(q+i)d)\equiv k(m+d)(\textrm{mod} \, m)$, 
such that $x_{1}$ and $t_{2}, \ldots , t_{e-1}$ are nonnegative integers. 
If $t_{i}>0$, for some $2\leq i\leq e-1$, we have 
$x_{1}(m+d)+\sum\limits_{i=2}^{e-1}t_{i}((q+1)m+(q+i)d)>k(m+d)$. This 
proves that $\beta_{1}\subset \textrm{Ap}(\Gamma_{(e,q,d)}(\mathcal{S}),m)$.
\medskip
  
To show that $\beta_{2}\subset \textrm{Ap}(\Gamma_{(e,q,d)}(\mathcal{S}),m)$, we 
proceed by induction on $k$. If $k=0$, we have $(q+1)m+(q+e-1)d\in \mathcal{S}\subset \textrm{Ap}(\Gamma_{(e,q,d)}(\mathcal{S}),m)$. Suppose that, for $k=0,\ldots j-1\leq q$, the element 
$k(m+d)+(q+1)m+(q+e-1)d\in \textrm{Ap}(\Gamma_{(e,q,d)}(\mathcal{S}),m) $. Let 
$$x_{1}(m+d)+\sum_{i=2}^{e-1}t_{i}((q+1)m+(q+i)d)\equiv j(m+d)+(q+1)m+(q+e-1)d(\textrm{mod} \, m)
.$$ 
If $x_{1}\geq j$, then 
$$(x_{1}-j)(m+d)+\sum_{i=2}^{e-1}t_{i}((q+1)m+(q+i)d)\equiv (q+1)m+(q+e-1)d(\textrm{mod} \, m)$$ 
and the element $(x_{1}-j)(m+d)+\sum_{i=2}^{e-1}t_{i}((q+1)m+(q+i)d)\in \Gamma_{(e,q,d)}(\mathcal{S})$.
Since  $(q+1)m+(q+e-1)d\in \textrm{Ap}(\Gamma_{(e,q,d)}(\mathcal{S}),m), $  
we get 
$$(x_{1}-j)(m+d)+\sum\limits_{i=2}^{e-1}t_{i}((q+1)m+(q+i)d)\geq (q+1)m+(q+e-1)d,$$ 
and we are done.
\medskip
 
If $0<x_{1} <j$, then
\begin{align*}\sum\limits_{i=2}^{e-1}t_{i}((q+1)m+(q+i)d)\equiv 
(j-x_{1})(m+d)+(q+1)m+(q+e-1)d(\textrm{mod} \, m).
\end{align*}
By induction hypothesis we have
$$(j-x_{1})(m+d)+(q+1)m+(q+e-1)d\in \textrm{Ap}(\Gamma_{(e,q,d)}(\mathcal{S}),m),$$ 
hence 
$\sum\limits_{i=2}^{e-1}t_{i}((q+1)m+(q+i)d)\geq (j-x_{1})(m+d)+(q+1)m+(q+e-1)d$. 
This proves that $\beta_{2}\subset \textrm{Ap}(\Gamma_{(e,q,d)}(\mathcal{S}),m)$. 
\medskip

If $x_{1}=0$, we have
\begin{align*}\sum\limits_{i=2}^{e-1}t_{i}((q+1)m+(q+i)d)\equiv j(m+d)+&(q+1)m+(q+e-1)d(\textrm{mod} \, m).
\end{align*}
Let us consider 
$\sum\limits_{i=2}^{e-1}t_{i}((q+1)m+(q+i)d)- [j(m+d)+(q+1)m+(q+e-1)d]$. This 
can be rewritten as 
$$(((\sum\limits_{i=2}^{e-1}t_{i})-1)(q+1)-j)m+(\sum\limits_{i=2}^{e-2}t_{i}(q+i)+(t_{e-1}-1)(q+e-1)-j)d,$$ 
which is clearly nonnegative, since $j\leq q+1$. 
Therefore $\beta_{2}\subset \textrm{Ap}(\Gamma_{(e,q,d)}(\mathcal{S}),m) $. 
Finally, $\beta_{3}\subset \mathcal{S} $ and therefore $\beta_{3}\subset \textrm{Ap}(\Gamma_{(e,q,d)}(\mathcal{S}),m) $. 
\medskip

Note that a maximal element of the Ap\'{e}ry set is 
$(q+1)(e+2q+d+1)+2q^2+(e+3+d)q+e+(e-1)d+1$. Therefore, the 
Frobenius number of $\Gamma_{(e,q,d)}(\mathcal{S})$ is $4q^2+(2e+2d+4)q+e(d+1)+1$.
\medskip
 
\noindent (2) The proof for proving that the set 
$\gamma_{1}\cup \gamma_{2}$ is the Ap\'{e}ry set for 
the numerical semigroup $\Gamma_{(e,q,d)}(\mathcal{T})$ with 
respect to $m$ is similar as in (i). We observe that, 
$\max \gamma_{1}=2q^2+(e+1)q+(q-\dfrac{e-4}{2}+e-3)d+\dfrac{e}{2}$ and $\max\gamma_{2}=(e+2q+d)(2q+1)$. 
Therefore $\max \gamma_{2}-\max \gamma_{1}=\left(2q^2+qd+\dfrac{3e}{2}+3q+2d\right)-\left(\dfrac{ed}{2}\right)>0$, since $q\geq e-4$ and $e\geq 4$. Hence, the Frobenius 
number of $\Gamma_{(e,q,d)}(\mathcal{T})$ is $(e+2q+d)2q+d$.\qed
\medskip

\begin{theorem}\label{symmetric} The numerical semigroups $\Gamma_{(e,q,d)}(\mathcal{S})$ and 
$\Gamma_{(e,q,d)}(\mathcal{T})$ are both symmetric.
\end{theorem} 

\proof (1) First we show that the Frobenius number of $\Gamma_{(e,q,d)}(\mathcal{S})$, 
which is $4q^2+(2e+2d+4)q+e(d+1)+1$, is odd. We claim that either $e$ or $d+1$ is even. 
If $e$ is odd, then $m=e+2q+1$ is even and, since $\gcd(m,d)=1$, we have that 
$d$ is odd. Therefore $d+1$ is even and our claim is proved. Hence $e(d+1)$ is even and 
therefore the Frobenius number $4q^2+(2e+2d+4)q+e(d+1)+1$ is odd.
\medskip

To show the symmetry of the semigroup $\Gamma_{(e,q,d)}(\mathcal{S})$, we calculate the 
genus, given by the formula 
$$g(\Gamma_{(e,q,d)}(\mathcal{S})) = \frac{1}{m} \left(\sum_{w\in \textrm{Ap}(\Gamma_{(e,q,d)}(\mathcal{S}),m)}w\right)-\frac{m-1}{2},$$ 
and prove that it equals 
$g(\Gamma_{(e,q,d)}(\mathcal{S}))=\frac{F(\Gamma_{(e,q,d)}(\mathcal{S}))+1}{2}$. 
Let us first calculate the sum of all elements in $\textrm{Ap}(\Gamma_{(e,q,d)}(\mathcal{S}),m)$.
\begin{eqnarray*}    
\sum_{w\in \textrm{Ap}(\Gamma_{(e,q,d)}(\mathcal{S}),m)}w &=&\sum_{w\in \beta_{1}} w + \sum_{w\in \beta_{2}} w + \sum_{w\in \beta_{3}} w\\
&=&\sum_{k=0}^{q+1} k(m+d)\\
&+& \sum_{k=0}^{q+1}k(m+d)+(q+1)m+(q+e-1)d\\
&+& \sum_{i=2}^{e-2}(q+1)m+(q+i)d\\
&=&\frac{(m+d)(q+1)+(q+2)}{2}\\
&+& (q+2)[(q+1)m+(q+e-1)d+\frac{(q+1)}{2}(m+d)]\\
&+& (e-3)[(q+1)m+(q+\frac{e}{2})d]\\
\end{eqnarray*}
\begin{eqnarray*}
g(\Gamma_{(e,q,d)}) &=& \frac{1}{m} \left(\sum_{w\in \textrm{Ap}(\Gamma_{(e,q,d)}(\mathcal{S}),m)}w\right) 
-\frac{m-1}{2}\\
&=&\frac{2(\sum_{w\in \textrm{Ap}(\Gamma_{(e,q,d)(\mathcal{S})},m)}w)-m(m-1)}{2m}
\end{eqnarray*}
Putting the value of $\sum_{w\in \textrm{Ap}(\Gamma_{(e,q,d)}(\mathcal{S}),m)}w$ 
and $m=e+2q+1$ in the expression for $g(\Gamma_{(e,q,d)}(\mathcal{S}))$, we get the desired 
relation between the genus and the Frobenius number. Hence the numerical semigroup 
$\Gamma_{(e,q,d)}(\mathcal{S})$ is symmetric.
\medskip

\noindent(2) The proof is similar as in (1). \qed
\medskip

\begin{lemma}\label{notsum} 
\begin{enumerate}
\item Let $n_{i}=(q+1)m+(q+i)d$, where $ 2\leq i\leq e-2$. 
There is no nonzero element $\alpha\in \mathrm{Ap}(\Gamma_{(e,q,d)}(\mathcal{S}),m)$ 
such that $\alpha+n_{i}=k(m+d)+(q+1)m+(q+e-1)d$, for every $1\leq k\leq q $.

\item Let $n_{i}=q(m+1)+(q-\dfrac{e-4}{2}+i)d+\dfrac{e}{2}$, where $0\leq i\leq e-3$. 
There is no nonzero element $\beta\in\mathrm{Ap}(\Gamma_{(e,q,d)}(\mathcal{T}),m)$ such that 
$\beta + n_{i}=k(m+d)$, for every $2\leq k\leq 2q$. 
\end{enumerate}
\end{lemma}

\proof (1)\, \textit{Case A.} Let $n_{i}+k'(m+d)=k(m+d)+(q+1)m+(q+e-1)d$, for 
some fixed $k\in \{1,\ldots,q\}$ and $k'\in \{1,\ldots,q+1\}$. From the above 
equation we have 
\begin{equation}\label{eq1}
(i+(k'-k)-e+1)d+(k'-k)m=0.
\end{equation}
\noindent Since $\gcd(m,d)=1,$ we have $d\mid(k'-k)$. Therefore $k'-k= \ell d$ and we get 
\begin{equation}\label{eq2}
i+\ell (d+m)-e+1=0.
\end{equation}
We consider the following possibilities:
\begin{enumerate}
\item[(a)] If $k'>k$, then $\ell>0$. Since $m>e$, $i+\ell (d+m)-e+1>0$ 
and Equation \ref{eq2} is not possible .

\item[(b)] If $k'=k$, then from Equation \ref{eq1} we get $i=e-1$, which is not possible.

\item[(c)] If $k'<k$, then $\ell<0.$ In this case, from Equation \ref{eq1} we get 
$i-\ell (d+m)-e+1=0.$ Since $i-\ell (d+m)-e+1<0$, this equation is not possible.
\end{enumerate}
\medskip

\noindent \textit{Case B.} Let  
\begin{eqnarray*}
{} & n_{i}+k'(m+d)+(q+1)m+(q+e-1)d\\
= & k(m+d)+(q+1)m+(q+e-1)d,
\end{eqnarray*}
for some fixed $k\in \{1,\ldots,q\}$ and $k'\in \{1,\ldots,q+1\}$. Since $\gcd(m,d)=1$, 
we have $d\mid (q+1+(k'-k))$. Therefore, $q+1+k'-k= \ell d$, and we get 
\begin{equation}\label{eq3}
q+i+\ell m+k'-k=0.
\end{equation}
We consider the following possibilities:
\begin{enumerate}
\item[(a)] If $k'>k$, then $q+i+\ell m+k'-k>0$. In this case Equation \ref{eq3} 
is not possible. 

\item[(b)] If $k'=k$, then from Equation \ref{eq3} we get $q+i+\ell m=0$, 
which is not possible.

\item[(c)] If $k'<k$, then from Equation \ref{eq3} we get $q+i+\ell m-(k'-k)=0$. 
This is not possible, since $\ell m>2q+1$. 
\end{enumerate}
\medskip

\noindent \textit{Case C.} Let $n_{j}=(q+1)m+(q+j)d$, where $2\leq j\leq e-2,\, i\neq j$. 
Let $n_{i}+n_{j}= k(m+d)+(q+1)m+(q+e-1)d$, for some fixed $k\in \{1,\ldots,q\}$. 
Since $\gcd(m,d)=1$, we have $d\mid (q+1-k)$. Therefore $q+1-k= \ell d$, and we get 
$\ell m+q+i+j-k-e+1=0$. Substituting $m=e+2q+1$ in the above equation we get 
\begin{equation}\label{eq4}
(\ell-1)e+(2\ell+1)q+\ell+i+j-k+1=0.
\end{equation}
Therefore $\ell\geq 1$, since $0\leq k\leq q$. We consider two possibilities:
\begin{enumerate}
\item[(a)] If $\ell>1$, then $(\ell-1)e+(2\ell+1)q+\ell+i+j-k+1>0.$ In this case equation \ref{eq4} is not possible. 
\item[(b)] If $\ell=1,$ then $(2\ell+1)q+\ell+i+j-k+1>0,$ which is not possible.
\end{enumerate}
\medskip

\noindent (2)\, The proof is similar to (1).\qed
\medskip

\begin{lemma}
\begin{enumerate}
\item Each element except the maximal element of the Ap\'{e}ry set 
$\textrm{Ap}(\Gamma_{(e,q,d)}(\mathcal{S}),m)$ has a unique 
expression.

\item Each element except the maximal element of the Ap\'{e}ry set 
$\textrm{Ap}(\Gamma_{(e,q,d)}(\mathcal{T}),m)$ has a unique expression.
\end{enumerate}
\end{lemma}

\proof 
\noindent (1) We have, $\textrm{Ap}(\Gamma_{(e,q,d)}(\mathcal{S}),m) =\beta_{1}\cup\beta_{2}\cup\beta_{3}$, 
where
\begin{itemize}
\item $\beta_{1}=\{k(m+d)\mid 0\leq k\leq q+1\}$,
\item $\beta_{2}=\{k(m+d)+(q+1)m+(q+e-1)d \mid 0\leq k\leq q+1\}$,
\item $\beta_{3}=\{(q+1)m+(q+i)d\mid 2\leq i\leq e-2\}$.
\end{itemize}
We have $k(m+d)<(q+1)m+(q+2)d$, for $0\leq k\leq q+1$. Therefore, each element 
of $\beta_{1}$ has a unique expression. Elements of $ \beta_{3}$ are in minimal generating set of $\Gamma_{(e,q,d)}(\mathcal{S})$, therefore each element 
of $\beta_{3}$ has a unique expression as well. Let
 \begin{itemize}
 \item $n_{1}=m+d$,
 \item $n_{i}=(q+1)m+(q+i)d$, \, $2\leq i\leq e-1$,
 \item  $m_{k}=k(m+d)+(q+1)m+(q+e-1)d$, \, for $1\leq k\leq q+1$.
\end{itemize}
Then for $1\leq k\leq q$, we have $m_{k} =kn_{1}+n_{e-1}$ and observe that 
$(m+d)\nmid m_{k} $ otherwise $(m+d)\mid (q+1)m+(q+e-1)d$, a contradiction. By lemma \ref{notsum}, $n_{i}+\alpha \neq m_{k}$, 
where $1\leq k\leq q$, \, $2\leq i\leq e-2$ and $\alpha$ is an element of 
$\textrm{Ap}(\Gamma_{(e,q,d)}(\mathcal{S}),m)$. We fix $k\in\{ 1,\ldots q\}$ and let 
$a_{i}=m_{k}-n_{i}$, for every $2\leq i\leq e-2$. If $a_{i}\in \Gamma_{(e,q,d)}(\mathcal{S})$, 
then there is an element $b_{i}\in \textrm{Ap}(\Gamma_{(e,q,d)}(\mathcal{S}), m)$ such that 
$a_{i}=b_{i}+lm$ for some $l\geq 0$. Then $m_{k}=n_{i}+b_{i}+lm$ for some 
$l\geq 0$. Since $m_{k}\in \textrm{Ap}(\Gamma_{(e,q,d)}(\mathcal{S}),m)$ we have $l=0$, 
hence $m_{k}=n_{i}+b_{i}$, which gives a contradiction since 
$b_{i}\in \textrm{Ap}(\Gamma_{(e,q,d)}(\mathcal{S}),m)$.
\medskip
 
\noindent (2) The proof is similar as in (1). \qed
\medskip

\begin{theorem}\label{mainthm}
The cardinality of a minimal presentation for both the numerical 
semigroups $\Gamma_{(e,q,d)}(\mathcal{S})$ and $\Gamma_{(e,q,d)}(\mathcal{T})$ is $\frac{e(e-1)}{2}-1$.
\end{theorem}

\proof Proof is essentially the same as in Proposition 7 and Proposition 8 in \cite{ros1}. \qed
\medskip

\noindent \textbf{Remark.} Our result support the conjecture that the symmetric condition on numerical semigroup put a bound on the cardinality of minimal presentation of numerical semigroup.
\bigskip

\section{Almost maximal concatenation}
Let $e\geq 4$ be an integer; $a=e+1$, $b > a+(e-3)d$, $\gcd(a,d)=1$ 
and $d\nmid(b-a)$. Let $M=\{a, a+d, a+2d,\ldots,a+(e-3)d, b,b+d\}$ and 
we assume that the set forms a minimal generating set for the numerical 
semigroup $\Gamma_{e}(M)$, generated by the set $M$. We once again 
recall that definition of concatenation includes the minimality of 
the sequence of integers generating the numerical semigroup $\Gamma_{e}(M)$. 
For example, for the choice of $e=d=4, a=5, b=10$, we get the 
concatenated sequence $5, 9, 10, 14$, which is not a good example due to 
its non-minimality. We start with a brief sufficient condition for minimality 
along with some supporting examples. 
\medskip

\begin{lemma}\label{minimalcondition}
Let $S=\{a,a+d,a+2d,\ldots,a+(a-4)d,b,b+d\}$, such that $\gcd(a,d)=1$, $b>a+(a-4)d$, 
$d\nmid (b-a)$. Let 
$d\equiv i\, (\mbox{mod}\, a)$. If $b\not\in \langle a, a+d,a+2d,\ldots, a+(a-4)d\rangle$ and 
$b\not\equiv (a-1)i\, (\mbox{mod}\, a)$, then the set $S$ is minimal. 

\proof  Let $b+d=\sum_{t=1}^{a-3} k_{t}(a+(t-1)d)+k_{a-2}b$. We consider the following cases 
and subcases to prove the statement.
\medskip

\noindent\textbf{Case1.} $k_{a-2}\neq 0$, then
$d=\sum_{t=1}^{a-3} k_{t}(a+(t-1)d)+(k_{a-2}-1)b$.
\medskip

\textbf{Subcase1.}If $k_{t}\neq 0$, for any $t\in\{2,\ldots,a-3\}$, then the value 
of the R.H.S is greater than that of the L.H.S.
\medskip

\textbf{Subcase2.} If $k_{t}= 0$, for every $t\in\{2,\ldots,a-3\}$ and $k_{a-2}=1$, 
then $d=k_{1}a$,  which is not possible. Again if $k_{a-2}>1$, then $d\geq b$, which 
is not possible.
\medskip

\noindent \textbf{Case 2.} If $k_{a-2}=0$, then $b+d=\sum_{t=1}^{a-3} k_{t}(a+(t-1)d)$.
\medskip

\textbf{Subcase1.} If $k_{t}\neq 0$, for any $t\in\{2,\ldots,a-3\}$ then 
$b\in \langle a, a+d,a+2d,\ldots, a+(a-4)d\rangle$, which is not possible.
\medskip

\textbf{Subcase2.} If $k_{t}= 0$ for every $t\in\{2,\ldots,a-3\}$, then $b=k_{1}a-d$, which gives $b\equiv (a-1)i\, (\mbox{mod}\, a)$, a contradiction to our assumption.\qed
\end{lemma}
\medskip

\begin{example}\label{minimalexample}
\begin{enumerate}
\item $a=6$, $d=25$, here $d\equiv 1\, (\mbox{mod}\, 6)$.
\begin{enumerate}
\item We take $b=81$, $81\not\in \langle 6,31,56\rangle$ and $81\equiv 3(\mbox{mod}\, 6)$. 
The set $\{6,31,56,81,106\}$ is minimal. 

\item We take $b=100$, $100\not\in \langle 6,31,56\rangle$ and $100\equiv 4(\mbox{mod}\, 6)$. 
The set $\{6,31,56,100,125\}$ is minimal. 
\end{enumerate}
\medskip

\item $a=6$, $d=191$, here $d\equiv 5\, (\mbox{mod}\, 6)$.
\begin{enumerate}
\item We take $b=459$, $459\not\in \langle 6,197,388\rangle$ and $459\equiv 3(\mbox{mod}\, 6)$. 
The set $\{6,197,388,459,650\}$ is minimal. 

\item We take $b=554$, $554\not\in \langle 6,197,388\rangle$ and $554\equiv 2(\mbox{mod}\, 6)$. 
The set $\{6,197,388,554,745\}$ is minimal. 
\end{enumerate}
\end{enumerate}
\end{example}
\medskip

We calculate the Ap\'{e}ry set, the Frobenius number and the pseudo-Frobenius 
numbers of $\Gamma_{e}(M)$. We also give a complete description of a minimal 
generating set of the defining ideal $\mathfrak{p}(M)$ of the 
affine monomial curve parametrized by $x_{0} = t^{a}$, $x_{1} = t^{a+d}, \ldots , 
x_{e-3}=t^{a+(e-3)d}$, $x_{e-2}= t^{b}$, $x_{e-1}= t^{b+d}$ and 
thereby show that the minimal number of generators for $\mathfrak{p}(M)$ is a 
bounded function of $e$.

\subsection{Ap\'{e}ry Set for $\Gamma_{e}(M)$}
\begin{theorem}\label{apery} Let $e,a,b,d, M$ be as above. Let us write 
$d\equiv i(\mathrm{mod}\, a)$. 
\begin{enumerate}[(1)]
\item If $e=4,$ then 
\begin{enumerate}[(i)]
\item $\mathrm{Ap}(\Gamma_{4}(M),5)= \{0, 5+d, b, b+d, 2b \} $, when $b\equiv 2i(\mathrm{mod}\,5)$

\item $\mathrm{Ap}(\Gamma_{4}(M),5)= \{0, 5+d, b, b+d,2(5+d)\}$, when $b\equiv 3i(\mathrm{mod}\, 5)$.
\end{enumerate}
\medskip

\item If $e\geq 5,$ then 
\begin{enumerate}[(i)]
\item $\mathrm{Ap}(\Gamma_{e}(M),a)=\{0, a+d, \ldots, a+(a-4)d, b, b+d, b+a+2d \} $, when 
$b\equiv(a-3)i(\mathrm{mod}\, a)$

\item $\mathrm{Ap}(\Gamma_{e}(M),a)=\{0, a+d, \ldots, a+(a-4)d, b, b+d,2a+(a-3)d\}$, when 
$b\equiv(a-2)i(\mathrm{mod}\,a)$
\end{enumerate}
\end{enumerate}
\end{theorem}

\proof\textbf{(1)} Suppose $e=4$. We have $\{0,5+d,b,b+d\}\subset \mathrm{Ap}(\Gamma_{4}(M),5)$. 
We have to find one extra element, say $s$,  in $\mathrm{Ap}(\Gamma_{4}(M),5)$. Suppose 
$s=m_{1}(5+d)+m_{2}b+m_{3}(b+d)$. We first note that, if $b \equiv 4i\, (\mathrm{mod}\, 5)$ 
then $(b+d) \equiv 0\, (\mathrm{mod}\, 5)$, which is not possible. We now 
consider the following cases:
\medskip

\noindent\textbf{(i)} If $b\equiv 2i\, (\mathrm{mod}\, 5)$, then 
$b+d\equiv 3i\, (\mathrm{mod}\, 5)$. Hence $s\equiv 4i\, (\mathrm{mod}\, 5)$. 
Therefore, $4\, \equiv (m_{1} + 2m_{2} + 3m_{3} )\, (\mathrm{mod}\, 5)$, and 
we get $(m_{1} + 2m_{2} + 3m_{3} )\geq 4 $, and hence $s\geq 2b$. 
Here we use the fact, $b\equiv 2(5+d)(\mathrm{mod}\, 5)$ and $b\in \mathrm{Ap}(\Gamma_{4}(M),5)$, therefore $b< 2(5+d)$. 
Since $s\in \mathrm{Ap}(\Gamma_{4}(M),5)$ and $s\equiv 2b\, (\mathrm{mod}\, 5)$, 
therefore $s=2b$.
\medskip

\noindent\textbf{(ii)} If $b \equiv 3i\, (\mathrm{mod}\, 5)$, then 
$(b+d)\, \equiv 4i\, (\mathrm{mod}\, 5)$, hence $s\equiv 2i\, (\mathrm{mod}\, 5)$. 
Therefore, $2\, \equiv (m_{1} + 3m_{2} + 4m_{3} )\, (\mathrm{mod}\, 5)$, and 
we get $(m_{1} + 3m_{2} + 4m_{3} )\geq 2 $. Hence, 
$s\geq 2(5+d)$. Moreover, $s\in \mathrm{Ap}(\Gamma_{4}(M),5)$ and $s\equiv 2(5+d)\, (\mathrm{mod}\, 5)$, 
therefore $s=2(5+d)$.
\medskip

\noindent\textbf{(2)} If $e\geq 5$, we assume that 
$s=\sum\limits_{k=1}^{a-4} m_{k}(a+kd)+m_{a-3}b+m_{a-2}(b+d)$ 
is the new element in $ \mathrm{Ap}(\Gamma_{e}(M),a)$. We have,
\medskip

\noindent\begin{tabular}{|c|c|c|c|c|c|c|c|c|c|}
\hline 
$\mathrm{Ap}(\Gamma_{e}(M),a)$ & $0$ & $a+d$ & $a+2d$ & $a+3d$ & $\cdots$ & $a+(a-4)d$ & $b$ & $b+d$ & $s$ \\ 
\hline 

Modulo Class & $0$ & $i$ & $2i$ & $3i$ & $\cdots$ & $(a-4)i$ & ? & ? & ?  \\ 
\hline 
\end{tabular} 
\medskip

We first note that $b \not\equiv (a-1)i\, (\mathrm{mod}\, a)$, 
otherwise $(b+d)\equiv 0\, (\mathrm{mod}\, a)$, which is not possible. 
We now consider the cases $b \equiv (a-3)i\, (\mathrm{mod}\, a)$ and $b\equiv (a-2)i\, (\mathrm{mod}\, a)$.
\medskip

\noindent\textbf{(i)} If $b \equiv (a-3)i\, (\mathrm{mod}\, a)$, then, $b+d\equiv (a-2)i\, (\mathrm{mod}\, a)$, and $s \equiv (a-1)i\, (\mathrm{mod}\, a)$. We also note 
that $(a+b+2d)\equiv (a-1)i\, (\mathrm{mod}\, a)$ and 
\begin{equation}\label{equn1}
s=\sum\limits_{k=1}^{a-4} m_{k}(a+kd)+m_{a-3}b+m_{a-2}(b+d) 
\end{equation}
where $m_{k}\geq 0$, for every $ 1\leq k\leq a-2$. Therefore,
\medskip

\noindent \begin{align*}(a-1)& \equiv\left(\sum\limits_{k=1}^{a-3}  m_{k}k+ (a-3)m_{a-3}+m_{a-2}(a-2)\right)\, 
(\mathrm{mod}\, a)
\end{align*} 
\smallskip

\noindent(A) If $m_{a-2}=0$, then from Equation \ref{equn1} we get 
\begin{equation}\label{equn2}
s=\sum\limits_{k=1}^{a-4} m_{k}(a+kd)+m_{a-3}b. 
\end{equation}

We discuss this case through the following subcases:
\begin{enumerate}[(a)]
\item Suppose that $m_{a-3}\neq 0$ and $\sum\limits_{k=1}^{a-4} m_{k}\neq 0$. 
At first we note that if $\sum\limits_{k=1}^{a-4} m_{k}\geq 2$, then 
$s = (a+k_{1}d)+(a+k_{2}d)+m_{a-3}b > a+b+2d$, $1\leq k_{1},k_{2}\leq a-4$, 
which is absurd. Similarly, if $m_{a-3}\geq 2$, then $s>a+b+2d$, that is, we have 
a contradiction. Therefore, $\sum\limits_{k=1}^{a-4} m_{k}=1$ and $m_{a-3}=1$.  Hence, 
$m_{j}=1$ for some $1\leq j\leq a-4$, and $m_{k}=0$ for all 
$1\leq k\leq a-4$ and $k\neq j$. Let $m_{1}=1=m_{a-3}$ and 
$m_{k}=0$ for every $2\leq k\leq (a-4)$. Then $s=a+d+b$ and therefore 
$s \equiv (a-2)i\, (\mathrm{mod}\, a)$, which is not possible. 
Let $m_{2}=1=m_{a-3}$ and $m_{k}=0$, for $k\neq 2,a-3$. Then $s=a+b+2d$. 
If $m_{j}=1=m_{a-3}$, $2<j\leq a-4$, and $m_{k}=0$ for $k\neq j,a-3$, 
then $s=a+kd+b>a+b+2d$, which is not possible.
\medskip

\item If $m_{a-3}\neq 0$ and $\sum\limits_{k=1}^{a-4} m_{k}= 0$, then, 
we must have $m_{a-3}\geq 2$ otherwise if $m_{a-3}=1$ then $s=b$ which 
is not possible. Therefore, the only possibility is $s\geq 2b>b+a+2d$.
\medskip

\item If $m_{a-3}=0$ and $\sum\limits_{k=1}^{a-4} m_{k}\neq 0$, then, 
we must have $\sum\limits_{k=1}^{a-4} m_{k}\geq 2$. Since, 
$(a-1)\equiv (\sum\limits_{k=1}^{a-4} m_{k}k)\, (\mathrm{mod}\, a)$, 
we have $\sum\limits_{k=1}^{a-4} m_{k}k=c_{1} a+(a-1)$, where $c_{1}\geq 0$. 
Substituting values in \ref{equn2} we get
$$s=a+(a-1)d+(\sum\limits_{k=1}^{a-4} m_{k}-1) a+c_{1}ad \geq 2a+(a-1)d.$$ 
Since $b\equiv 2a+(a-3)d \, (\mathrm{mod}\, a)$ and 
$$2a+(a-3)d=(a+(a-4)d)+(a+d)\in\Gamma_{e}(M),$$ 
we have $b\leq 2a+(a-3)d$. If 
$$b=2a+(a-3)d=(a+(a-4)d)+(a+d),$$ 
then it contradicts the fact that $M$ is a minimal generating set for $\Gamma_{e}(M)$. 
Therefore, $b<2a+(a-3)d$. If $a+(a-3)d<b$, then we have a contradiction. Therefore, 
$b\leq a+(a-3)d$. But $b=a+(a-3)d$ implies that $b+d=a+(a-4)d$, which is not possible. 
Hence $b< a+(a-3)d $, thus $b+a+2d< 2a+(a-1)d \leq s$. 
\end{enumerate}

\noindent(B) If $m_{a-2}\neq 0$ then we have following  subcases: 
\begin{enumerate}[(a)]
\item  If $m_{a-3}\neq 0$ and $\sum\limits_{k=1}^{a-4} m_{k}\neq 0$, then 
$s\geq a+d+b+b+d>a+b+2d$.
\item If $m_{a-3}\neq 0$ and $\sum\limits_{k=1}^{a-4} m_{k}= 0$, then 
$s\geq b+b+d>a+b+2d$.
\item If $m_{a-3}= 0$ and $\sum\limits_{k=1}^{a-4} m_{k}\neq 0$, then 
$s\geq a+d+b+d=a+b+2d$. 
\item If $m_{a-3}= 0$ and $\sum\limits_{k=1}^{a-4} m_{k}= 0$, then 
$s=m_{a-2}(b+d)$. If $m_{a-2}=1$ then $s=b+d$, which is not possible. 
Therefore, we must have $m_{a-2}\geq 2$ and the only possibility is 
$s\geq 2(b+d)>a+b+2d$.
\end{enumerate}
\medskip

\noindent \textbf{(ii)} If $b\equiv (a-2)i\, (\mathrm{mod}\, a)$, then $b+d\equiv (a-1)i\, (\mathrm{mod}\, a)$. Therefore, $s\equiv (a-3)i\, (\mathrm{mod}\, a)$. We note that 
$2a+(a-3)d\equiv (a-3)i\, (\mathrm{mod}\, a)$. Using Equation \ref{equn1} we get,
$$(a-3)\equiv \left(\sum\limits_{k=1}^{a-4} k m_{k}+ (a-2)m_{a-3}+ (a-1)m_{a-2}\right)\, (\mathrm{mod}\, a).$$

\noindent(A) Suppose that $m_{a-2}=0$.
\begin{enumerate}[(a)]
\item If $m_{a-3}\neq 0$ and $\sum\limits_{k=1}^{a-4} m_{k}\neq 0$, 
then $ s\geq a+d+b>2a+(a-3)d$.
\item Suppose that $m_{a-3}=0$ and $\sum\limits_{k=1}^{a-4} m_{k}\neq 0$. 
If $\sum\limits_{k=1}^{a-4} m_{k}=1$, then $s=a+jd$ for some $1\leq j\leq (a-4)$, 
which is not possible. Hence, $\sum\limits_{k=1}^{a-4} m_{k}\geq 2$. 
Since $(a-3)\equiv(\sum\limits_{k=1}^{a-4} k m_{k})(\mathrm{mod}\, a)$, 
therefore, $\sum\limits_{k=1}^{a-4} k m_{k}=c_{1} a+(a-3)$ for some $c_{1}\geq 0$. 
Substituting the values we get 
$$ s=a+(a-3)d+a(\sum\limits_{k=1}^{a-4} m_{k}-1) +c_{1}ad \geq 2a+(a-3)d.$$

\item If $\sum\limits_{k=1}^{a-4} m_{k}= 0$ and $m_{a-3}\neq 0$, then  $s=m_{a-3}b$. 
If $m_{a-3}=1$, then $s=b$ and $s\equiv (a-3)(\mathrm{mod}\, a)$, which is a contradiction. Therefore, we must have $m_{a-3}\geq 2$, hence $s\geq 2b>2a+(a-3)d$. 
\end{enumerate}
\medskip

\noindent(B) If $m_{a-2}\neq 0,$ then we can have following subcases:
\begin{enumerate}[(a)]
\item If $m_{a-3}\neq 0$ and $\sum\limits_{k=1}^{a-4} m_{k}\neq 0,$ then $s\geq a+d+b+b+d> 2a+(a-3)d$.
\item If $m_{a-3}= 0$, and $\sum\limits_{k=1}^{a-4} m_{k}\neq 0,$ then $s\geq a+d+b+d> 2a+(a-3)d$.
\item If $\sum\limits_{k=1}^{a-4} m_{k}= 0$ and $m_{a-3}\neq 0,$ then $s\geq b+b+d>2a+(a-3)d$.
\item If $m_{a-3}= 0$ and $\sum\limits_{k=1}^{a-4} m_{k}= 0.$ Since $s=b+d$ is not possible, therefore  we must have $m_{a-2}\geq 2$, hence $s\geq 2(b+d)>2a+(a-3)d$.\qed
\end{enumerate}
\medskip

\begin{theorem} The Frobenius number of the numerical semigroup 
$\Gamma_{e}(M)$ is as follows:
\begin{enumerate}[(1)]
\item When $e=4$, then
\[\left\{\begin{array}{l}
F(\Gamma_{4}(M))= 2b-5,\quad \mathrm{if}\,  b\equiv 2i(\mathrm{mod}\, 5);\\[2mm]
F(\Gamma_{4}(M))=5+2d,\quad \mathrm{if} \, b\equiv 3i(\mathrm{mod}\, 5),
\end{array} \right.\] where $d\equiv i(\mathrm{mod}\, 5)$.

\item When $e\geq 5$, then
\[\left\{\begin{array}{l}
F(\Gamma_{e}(M))= b+2d,\quad\quad\quad\quad\mathrm{if}\,  b\equiv(a-3)i(\mathrm{mod}\, a);\\[2mm]
F(\Gamma_{e}(M))=a+(a-3)d,\quad\, \, \mathrm{if}\,  b\equiv(a-2)i(\mathrm{mod}\, a)\, 
\mathrm{and}\\
\quad\quad\quad\quad\quad\quad\quad\quad\quad\quad\quad\quad\quad  b-a\leq a+(a-4)d;\\[2mm]
F(\Gamma_{e}(M))=b+d-a,\quad\, \, \, \, \quad  \mathrm{if}\,  b\equiv(a-2)i(\mathrm{mod}\, a)\,  \mathrm{and}\\
\quad\quad\quad\quad\quad\quad\quad \quad\quad\quad\quad\quad\quad b-a> a+(a-4)d,\\
\end{array} \right.\] where $d\equiv i(\mathrm{mod}\, a)$.
\end{enumerate}
\end{theorem}

\begin{proof} The result follows by an application of the \cite[Proposition 2.12]{rgs}, 
which says that $F(\Gamma_{e}(M)) = (\mathrm{max}\mathrm{Ap}(\Gamma_{e}(M), a)) - a$. 
Note that, if $b\equiv (a-2)i(\mathrm{mod}\, a)$, if $b-a\leq a+(a-4)d$, 
then the maximum element of the Ap\'{e}ry set is $2a+(a-3)d$ and if 
$b-a>a+(a-4)d$ then the maximum element of the Ap\'{e}ry set is $b+d$.
\end{proof}
\medskip

\begin{definition}
Let $\Gamma$ be a numerical semigroup. We say thet $x\in\mathbb{Z}$ is a 
pseudo-Frobenius number of $\Gamma$ if $x\notin \Gamma$ and 
$x+s\in \Gamma$ for all $s\in \Gamma\setminus \{0\}$. We 
denote by $\mathbf{PF}(\Gamma)$ the set of all pseudo-Frobenius 
numbers of $\Gamma$. The cardinality of $\mathbf{PF}(\Gamma)$ 
is denoted by $t(\Gamma)$ and it is called the type of $\Gamma$.
\end{definition}
\medskip

Let $a,b\in \mathbb{Z}$. We define $\leq_{\Gamma}$ as follows: 
$a\leq_{\Gamma} b$ if $b-a\in \Gamma$. With this order relation, 
$\mathbb{Z}$ becomes a poset.
\medskip

\begin{theorem}\label{maximals}
Let $\Gamma$ be a numerical semigroup and $\mathbf{a}\in  \Gamma\setminus \{0\}$. Then 
$\mathbf{PF}(\Gamma)=\{w-\mathbf{a}\mid w\in \,\mathrm{Maximals}_{\leq_{\Gamma}}Ap(\Gamma, \mathbf{a})\}$. 
\end{theorem}

\proof See proposition $8$ in \cite{ap}.\qed
\medskip

\begin{theorem}\label{pseudofrob} Let $e,a,b,d, M$ be as in \ref{apery}. Let us write 
$d\equiv i(\mathrm{mod}\, a)$. 
\begin{enumerate}[(1)]
\item If $e=4,$ then, 
\begin{enumerate}[(i)]
\item $\mathbf{PF}(\Gamma_{4}(M))= \{d, b+d-5, 2b-5 \} $, when $b\equiv 2i(\mathrm{mod}\,5)$;

\item $\mathbf{PF}(\Gamma_{4}(M))= \{ b-5, b+d-5,(5+2d)\}$, when $b\equiv 3i(\mathrm{mod}\, 5)$.
\end{enumerate}
\medskip

\item If $e\geq 5$, then, 
\begin{enumerate}[(i)]
\item $\mathbf{PF}(\Gamma_{e}(M))=\{3d, \ldots, (a-4)d, b+2d \} $, when 
$b\equiv(a-3)i(\mathrm{mod}\, a)$;

\item $\mathbf{PF}(\Gamma_{e}(M))=\{b-a, b+d-a, a+(a-3)d\}$, when 
$b\equiv(a-2)i(\mathrm{mod}\,a)$
\end{enumerate}
\end{enumerate}

\end{theorem}
\proof We prove statement (2); proof of statement (1) is similar. We know that 
if $\Gamma$ is a numerical semigroup minimally generated by 
$\{m_{1},\ldots,m_{e}\}$, with $m_{1}<\cdots <m_{e}$, then 
$\{m_{1},\ldots,m_{e}\}\subset \mathrm{Ap}(\Gamma,m_{1})$. Moreover, minimality 
of $\{m_{1},\ldots,m_{e}\}$ ensures that, for any $j>k$, we have 
$m_{k}\nleq_{\Gamma} m_{j}$. 
\medskip

\noindent\textbf{Case 1.} Let $b\equiv(a-3)i(\mathrm{mod}\,a)$. Here we note that 
$t\leq_{\Gamma_{e}(M)}(b+a+2d)$, for $t\in\{0,a+d,a+2d,b,b+d\}$. If 
$a+jd\leq_{\Gamma_{e}(M)}(b+a+2d)$ for some $j\in\{3,\ldots,a-4\}$, then 
$(b+a+2d)-(a+jd)=b-(j-2)d\in \Gamma_{e}(M)$. We have 
$b-(j-2)d\equiv(a-1-j)i(\mathrm{mod}\,a)$, since $a+(a-1-j)d\in \mathrm{Ap}(\Gamma_{e}(M),a)$ and it has the same modulo class with $b-(j-2)d$. Therefore 
$b-(j-2)d=ka+[a+(a-1-j)d]$, for some $k\geq 0$. If $k\geq 1$, then 
$$b=(k-1)a+[a+(j-2)d)+(a+(a-1-j)d],$$ 
which is a contradiction as $b$ is an element of a minimal generating set 
of $\Gamma_{e}(M)$. If $k=0$, then $b-a=(a-3)d$, that is $d\mid b-a$, 
a contradiction. By theorem \ref{maximals} and this observation we have 
completed the proof of this case.
\medskip

\noindent\textbf{Case 2.} Let $b\equiv(a-2)i(\mathrm{mod}\,a)$. At first we observe that 
$t\leq_{\Gamma_{e}(M)}[2a+(a-3)d]$, for $t\in\{0,a+d,\ldots,a+(a-4)d\}$. We 
now consider four subcases.
\medskip

\textbf{Subcase 1.} Let $b+d>2a+(a-3)d$. If $(b+d)-[2a+(a-3)d)]\in \Gamma_{e}(M)$, 
then $(b+d)-[2a+(a-3)d)]=ka+a+2d$, for some $k\geq 0$,  since 
$(b+d)-[2a+(a-3)d)]\equiv 2i(\mathrm{mod}\,a)$ and $a+2d$ is an element 
of $\mathrm{Ap}(\Gamma_{e}(M),a)$ with the same modulo class. 
We have, 
$$b+d=(k+1)a+(a+3d)+[a+(a-4)d],$$ 
which is a contradiction to the fact that $b+d$ is an element of a 
minimal generating set of $\Gamma_{e}(M)$.
\medskip

\textbf{Subcase 2.} Let $b>2a+(a-3)d$. If $b-[2a+(a-3)d)]\in \Gamma_{e}(M)$, 
we have $b=(k+1)a+(a+2d)+[a+(a-4)d]$, for some $k\geq 0$, which gives a contradiction 
to the fact that $b$ is an element of a minimal generating set of $\Gamma_{e}(M)$.
\medskip

\textbf{Subcase 3.} Let $b+d<2a+(a-3)d $. Suppose that 
$[2a+(a-3)d]-(b+d)=2a+(a-4)d-b\in \Gamma_{e}(M)$. Then, 
$2a+(a-4)d-b=ka+b$, for some $k\geq 0$, since 
$b\equiv[2a+(a-4)d-b]\equiv(a-2)i(\mathrm{mod}\,a)$ and it is 
an element of $\mathrm{Ap}(\Gamma_{e}(M),a)$. If $k>0$, then 
$a+(a-4)d=(k-1)a+2b$ contradicts the minimality of generators of 
$\Gamma_{e}(M)$. If $k=0$, then $2a+(a-4)d=2b$. On the 
other hand, $b>a$ and $b>a+(a-4)d$\, 
imply that\,  $2b>2a+(a-4)d$, which is a contradiction. 
\medskip

\textbf{Subcase 4.} Let $b<2a+(a-3)d $. Suppose that 
$2a+(a-3)d-b\in \Gamma_{e}(M)$. Then, 
$2a+(a-3)d-b=ka+(b+d)$ for some $k\geq 0$, since 
$b+d\equiv[2a+(a-3)d-b]\equiv(a-1)i(\mathrm{mod}\,a)$ 
and it is an element of $\mathrm{Ap}(\Gamma_{e}(M),a)$. 
If $k>0$, then $a+(a-4)d=(k-1)a+2(b+d)$ contradicts the 
minimality of generators of $\Gamma_{e}(M)$. If $k=0$, 
then $2a+(a-3)d=b+(b+d)=2b+d$. On the other hand, $b+d > a+d$ and 
$b>a+(a-4)d$\, imply that\, $2b+d>2a+(a-3)d$, which is a contradiction.\qed
\medskip

\subsection{Minimal generating set for the defining ideal} 
First let us handle the case $e=4$. Let $\Gamma$ be a 
numerical semigroup minimally generated by $\{n_{0},n_{1},n_{2},n_{3}\}$. 
Let $\phi: \mathbb{Z}^{4}_{\geq 0}\rightarrow \Gamma$ be a monoid epimorphism 
defined by 
$$\phi(a_{0}\epsilon_{0}+a_{1}\epsilon_{1}+a_{2}\epsilon_{2}
+a_{3}\epsilon_{3})=a_{0}n_{0}+a_{1}n_{1}+a_{2}n_{2}+a_{3}n_{3}.$$
Let $T=\{(a_{1},a_{2},a_{3})\in \mathbb{Z}^{3}_{\geq 0}\, \mid a_{1}n_{1}+a_{2}n_{2}+a_{3}n_{3}\notin Ap(\Gamma, n_{0})\}$ and
$$\mathrm{minimals}\, (T)=\{\alpha_{1}=(\alpha_{11},\alpha_{12},\alpha_{13}),\ldots, \alpha_{t}=(\alpha_{t1},\alpha_{t2},\alpha_{t3})\};$$ 
with respect to the lexicographic ordering on $\mathbb{Z}^{4}_{\geq 0}$. 
For every $i\in \{1,\ldots,t\}$, we define 
$x_{i}=0 \epsilon_{0}+\alpha_{i1}\epsilon_{1}+\alpha_{i2}\epsilon_{2}
+\alpha_{i3}\epsilon_{3}\in \mathbb{Z}^{4}_{\geq 0}$. 
We have $\phi(x_{i}) \notin Ap(\Gamma, n_{0})$, by Theorem 1 in \cite{ros2}. 
There exist $(\beta_{i0}\ldots, \beta_{i3})\in \mathbb{Z}^{4}_{\geq 0}$, 
with $\beta_{i0}\neq 0$, such that $\phi(x_{i})=\beta_{i0}n_{0}+\beta_{i1}n_{1}+\beta_{i2}n_{2}+\beta_{i3}n_{3}$. For every $i\in \{1,\ldots,t\}$, define 
$y_{i}=\beta_{i0}\epsilon_{0}+\beta_{i1}\epsilon_{1}+\beta_{i2}\epsilon_{2}+\beta_{i3}\epsilon_{3}\in \mathbb{Z}^{4}_{\geq 0}$. It is clear that $\phi(x_{i})=\phi(y_{i})$.
\medskip

It is known by the first part of Theorem 1 in \cite{rgs1} that the cardinality 
of a minimal presentation of $\Gamma_{4}(M)$ is $6$. We now explicitly 
compute a minimal presentation of $\Gamma_{4}(M)$ and subsequently compute a 
minimal generating set for the defining ideal $\mathfrak{p}(M)$.
\medskip

\begin{lemma}
Suppose that every element of $\mathrm{Ap}(\Gamma, n_{0})$ has a unique expression, 
then the set $\rho =\{(x_{1},y_{1}),\ldots,(x_{t},y_{t})\}$ is a minimal presentation 
of $\Gamma$.
\end{lemma}

\proof See Theorem 1 in \cite{ros2}.\qed 
\medskip

\begin{lemma}\label{mg4}
Let $e=4$. Write $d\equiv i(\mathrm{mod}\, a)$. 
\begin{enumerate}[(1)]
\item If $b\equiv 2i(\mathrm{mod} \, 5)$, then $$\rho_{1} =\{(2\epsilon_{1},y_{1}),(3\epsilon_{2},y_{2}),  (2\epsilon_{3},y_{3}),(\epsilon_{1}+\epsilon_{2},y_{4}),(\epsilon_{1}+\epsilon_{3},y_{5}), (\epsilon_{2}+\epsilon_{3},y_{6})\}$$ is a minimal presentation of $\Gamma_{4}(M)$.
\item If $b\equiv 3i(\mathrm{mod} \, 5)$, then $$\rho_{2} =\{(3\epsilon_{1},y_{1}),(2\epsilon_{2},y_{2}),  (2\epsilon_{3},y_{3}),(\epsilon_{1}+\epsilon_{2},y_{4}),(\epsilon_{1}+\epsilon_{3},y_{5}), (\epsilon_{2}+\epsilon_{3},y_{6})\}$$ is a minimal presentation of $\Gamma_{4}(M)$.
\end{enumerate}
Therefore, every element of $\mathrm{Ap}(\Gamma_{4}(M),5)$ has a unique expression. 
\end{lemma}

\proof It is enough to show that $2b$ has a unique expression.
\medskip

\noindent (1) \, Let $b\equiv 2i(\mathrm{mod} \, 5)$. Then 
$2b=m_{0} 5+ m_{1} (5+d)+ m_{2}b+m_{3}(b+d)$, such that 
$m_{0},m_{1}\geq 0$, $0\leq m_{2}\leq 2$, $0\leq m_{3}\leq 1$. 
\medskip

\begin{enumerate}[(i)] 
\item If $m_{2}=0$, then $2b =m_{0} 5+m_{1}(5+d)+(b+d)$, 
which implies that $b=(m_{0}-1)5+(m_{1}+1)(5+d)$. This 
is not possible because the generating set is minimal.

\item If $m_{2}=1$, then $b =m_{0} 5+m_{1}(5+d)+m_{3}(b+d)$. It is clear 
from this equation that $m_{3}$ must be $0$. Therefore, 
$b=m_{0} 5+m_{1}(5+d)$, which is not possible.
\end{enumerate}
\medskip

\noindent From (i) and (ii) it is clear that $2b$ is uniquely expressed 
in terms of generators of the numerical semigroup. In this case, 
$$\mathrm{minimals}(T)=\{(2,0,0),(0,3,0),(0,0,2),(1,1,0),(1,0,1),(0,1,1)\}.$$
Therefore, a minimal presentation $\rho_{1}$ of $\Gamma_{4}(M)$ 
is given by 
$$\rho_{1} =\{(2\epsilon_{1},y_{1}),(3\epsilon_{2},y_{2}),  (2\epsilon_{3},y_{3}),(\epsilon_{1}+\epsilon_{2},y_{4}),(\epsilon_{1}+\epsilon_{3},y_{5}), (\epsilon_{2}+\epsilon_{3},y_{6})\},$$ 
where $\phi(2\epsilon_{1})=\phi(y_{1})$, 
$\phi(3\epsilon_{2})=\phi(y_{2})$, $\phi(2\epsilon_{3})=\phi(y_{3})$, 
$\phi(\epsilon_{1}+\epsilon_{2})=\phi(y_{4})$, 
$\phi(\epsilon_{1}+\epsilon_{3})=\phi(y_{5})$ and 
$\phi(\epsilon_{2}+\epsilon_{3})=\phi(y_{6})$. 
\medskip

\noindent (2) \, Let $b\equiv 3i(\mathrm{mod} \, 5)$. Then 
$2(5+d)=m_{0} 5+ m_{1} (5+d)+ m_{2}b+m_{3}(b+d)$, where 
$m_{0}\geq 0$, $0\leq m_{1}\leq 2$, since $(5+2d)\leq b$. Therefore 
$0\leq m_{2}\leq 1$ and $m_{3}= 0$.

\begin{enumerate}[(i)] 
\item If $m_{2}=1$, then $2(5+d) = m_{0} 5+m_{1}(5+d)+m_{2}b$, which 
implies that $10-(m_{0}+m_{1})5=(m_{1}-2)d+b$. Now the right hand side 
of the expression is $(m_{1}-2)d+b > (m_{1}-2)d+(5+2d)=m_{1}d+5$. Therefore, the 
left hand side of the expression is $10-(m_{0}+m_{1})5>m_{1}d+5$, 
which is not possible.

\item If $m_{2}=0$, then $10+2d =m_{0} 5+m_{1}(5+d)$. We consider 
two subcasses: $m_{1}=0$ and $m_{1}=1$. If $m_{1}=0$ then 
$10+2d =m_{0} 5$ which implies $2d=5(m_{0}-2)$ which implies 
$ d\, \mathrm{divides}\, 5$, which is not possible. If $m_{1}=1$ 
then $10+2d =m_{0} 5+(5+d)$, which implies that $d=(m_{0}-1)5$. 
Therefore $d \mid 5$, which is not possible.
\end{enumerate}
\medskip

\noindent From (i) and (ii) it is clear that $2(5+d)$ is uniquely expressed in the terms 
of the generators of the numerical semigroup. In this case we have
$$\mathrm{minimals}(T)=\{(3,0,0),(0,2,0),(0,0,2),(1,1,0),(1,0,1),(0,1,1)\}.$$
Therefore, a minimal presentation of $\Gamma_{4}(M)$ is given by  
$$\rho_{2} =\{(3\epsilon_{1},y_{1}),(2\epsilon_{2},y_{2}),  (2\epsilon_{3},y_{3}),(\epsilon_{1}+\epsilon_{2},y_{4}),(\epsilon_{1}+\epsilon_{3},y_{5}), (\epsilon_{2}+\epsilon_{3},y_{6})\},$$ 
where $\phi(3\epsilon_{1})=\phi(y_{1})$, 
$\phi(2\epsilon_{2})=\phi(y_{2})$, 
$\phi(2\epsilon_{3})=\phi(y_{3})$, 
$\phi(\epsilon_{1}+\epsilon_{2})=\phi(y_{4})$, 
$\phi(\epsilon_{1}+\epsilon_{3})=\phi(y_{5})$ and 
$\phi(\epsilon_{2}+\epsilon_{3})=\phi(y_{6})$. 
Therefore, the cardinality of a minimal presentation 
of $\Gamma_{4}(M)$ is indeed $6$.\qed
\medskip

Let $\mathfrak{p}(M)$ be the defining ideal for $\Gamma_{e}(M)$. By the 
first part of Theorem 1 in \cite{rgs1}, we have 
$\mu(\mathfrak{p}(M))= \displaystyle\frac{e(e-1)}{2}-1$. Now we compute 
a minimal generating set for $\mathfrak{p}(M)$. We use the results in \cite{patil} 
to find a minimal generating set for 
$\mathfrak{p}(M)$. Let $\varepsilon_{0}:=(1,0,\ldots,0), \varepsilon_{1}:=(0,1,0\ldots,0), \ldots , \varepsilon_{p}:=(0,\ldots,0,1)$ denote the standard basis of ${(\mathbb{Z}_{\geq 0})}^{p+1}$. 
For $\alpha:=\sum_{i=0}^{i=p} \alpha_{i}\epsilon_{i}$, let us write 
$\deg(\alpha):=\sum_{i=0}^{i=p} \alpha_{i}n_{i}$. For example, if 
$\Gamma=\langle 11, 12, 13, 32, 53\rangle$ be a numerical semigroup and 
$\alpha=(2,3,6,7,8)\in {(\mathbb{Z}_{\geq 0})}^{5}$, then 
$$\deg(\alpha)=2\cdot 11 +3\cdot 12 + 6\cdot 13 + 7\cdot 32 +8\cdot 53.$$
Let us recall some definitions and a Lemma from \cite{patil}, which we 
require for describing a generating set for $\mathfrak{p}(M)$.
\medskip

\begin{definition}
For $s\in \Gamma$, let $\tau(s)$ be the unique maximal element of ${(\mathbb{Z}_{\geq 0})}^{p+1}$ of degree $s$ with 
respect to the lexicographic order. 
\vspace{2mm}
\begin{itemize}
\item We define 
$\underline{\mathcal{B}}:=
\underline{\mathcal{B}}(n_{0},\ldots , n_p)=\{\tau(s)\mid s\in \Gamma\}.$

\item For $\alpha:=\sum_{i=0}^{i=p} \alpha_{i}\epsilon_{i}\in {(\mathbb{Z}_{\geq 0})}^{p+1}$,
we put $x^{\alpha}:=\prod_{i=0}^{p} x^{\alpha_{i}}$. For $\alpha \in \eta$ and $0\leq i \leq p$, let 
$f(\alpha,i):= x_{i}x^{\alpha}-x^{\tau(s)}$, where $s:=\deg(\alpha+\varepsilon_{i})$.

\item For $1\leq i\leq p$, let 
$${\underline{\mathcal{B}}}_{i}:=\{\tau(s)\in \underline{\mathcal{B}}\mid \tau+\varepsilon\notin {\underline{\mathcal{B}}}\}$$
$${\underline{\mathcal{B'}}}_{i}:={\underline{\mathcal{B}}}_{i}\backslash \cup_{i=1}^{p}({\underline{\mathcal{B}}}_{i}+\varepsilon_{j}).$$
\end{itemize}
\end{definition}

\noindent We now state Lemma 3.6 from \cite{patil}.

\begin{lemma}\label{patillemma} Let $H=\langle m_{0},\ldots,m_{e-1}\rangle$ be a numerical semigroup.
Let $i,j\in I_{t} = \{p\in\mathbb{Z} \mid 0\leq p \leq t\}$. Then:
\begin{enumerate}
\item[(i)] $f(\tau(h),0)=0$ for every $h\in H$. 
\item[(ii)]$f(\lambda \epsilon_{0}+\tau,i)=X_{0}^{\lambda}f(\tau(h),i)$ for every $\tau\in \underline{\mathcal{B}}$ and $\lambda\in \mathbb{Z}^{+}$.
\item[(iii)] $f(\tau,i)=0$ for every $\tau\in \underline{\mathcal{B}}\setminus\underline{\mathcal{B}_{i}}$.
\item[(iv)] Either $f(\epsilon_{j},i)=0$ or $\epsilon_{j}\in \underline{\mathcal{B}_{i}^{'}} $.
\item[(v)] Let $\tau\in \underline{\mathcal{B}_{i}}$, $\tau^{'}\in \underline{\mathcal{B}_{j}}$ with $\tau+\epsilon_{i}=\tau^{'}+\epsilon_{j}$. Then: $f(\tau,i)=f(\tau^{'},j)$.
\end{enumerate}
\end{lemma}
\medskip

\begin{example}\label{example=5}
Let us consider the numerical semigroup $\Gamma_{5}=\langle 6,13,20,21,28\rangle$. 
Using Theorem \ref{apery} we get $\mathrm{Ap}(\Gamma_{5},6)=\{0,13,20,21,28,41\}$. 
Since $\Gamma_{5}$ is minimally generated by $M=\{6,13,20,21,28\}$, each element of 
$M$ has unique expression with respect to $\Gamma_{5}$. Hence $\tau(0)=\mathbf{0}$, 
$\tau(13)=\varepsilon_{1}$, $\tau(20)=\varepsilon_{2}$, $\tau(21)=\varepsilon_{3}$, 
$\tau(28)=\varepsilon_{4}$. Now $41=13+28=20+21$, since 
$\varepsilon_{1}+\varepsilon_{4}>\varepsilon_{2}+\varepsilon_{3}$ with respect to 
lexicographic order, we have $\tau(41)=\varepsilon_{1}+\varepsilon_{4}$. Therefore, 
$\underline{\mathcal{B}} = \{\mathbf{0},\varepsilon_{1},\varepsilon_{2},\varepsilon_{3},\varepsilon_{4},\varepsilon_{1}+\varepsilon_{4}\}$, $\underline{\mathcal{B}_{1}} = \{\varepsilon_{1},\varepsilon_{2},\varepsilon_{3},\varepsilon_{1}+\varepsilon_{4}\}$, 
$\underline{\mathcal{B}_{2}} = \{\varepsilon_{1},\varepsilon_{2},\varepsilon_{3},\varepsilon_{4},\varepsilon_{1}+\varepsilon_{4}\}$, $\underline{\mathcal{B}_{3}} = \{\varepsilon_{1},\varepsilon_{2},\varepsilon_{3},\varepsilon_{4},\varepsilon_{1}+\varepsilon_{4}\}$, 
$\underline{\mathcal{B}_{4}} = \{\varepsilon_{2},\varepsilon_{3},\varepsilon_{4},\varepsilon_{1}+\varepsilon_{4}\}$. We also have 
\begin{align*}
\bigcup_{j=0}^{e-1}({\underline{\mathcal{B}}}_{i}+\varepsilon_{j})
 =  \{\varepsilon_{0}+\varepsilon_{1},\varepsilon_{0}+\varepsilon_{2},\varepsilon_{0}+\varepsilon_{3},\varepsilon_{0}+\varepsilon_{1}+\varepsilon_{4}\}\\
\cup \{2\varepsilon_{1},\varepsilon_{1}+\varepsilon_{2},\varepsilon_{1}+\varepsilon_{3},2\varepsilon_{1}+\varepsilon_{4}\} \cup 
\{\varepsilon_{1}+\varepsilon_{2},2\varepsilon_{2},\varepsilon_{2}+\varepsilon_{3},\varepsilon_{1}+\varepsilon_{2}+\varepsilon_{4}\}\\
\cup \{\varepsilon_{1}+\varepsilon_{3},\varepsilon_{2}+\varepsilon_{3},2\varepsilon_{3},\varepsilon_{1}+\varepsilon_{3}+\varepsilon_{4}\} \cup \{\varepsilon_{1}+\varepsilon_{4},\varepsilon_{2}+\varepsilon_{4},\varepsilon_{3}+\varepsilon_{4},\varepsilon_{1}+2\varepsilon_{4}\}.
\end{align*}
Therefore, $\underline{\mathcal{B'}}_{1} = \{\varepsilon_{1},\varepsilon_{2},\varepsilon_{3}\}$, 
$\underline{\mathcal{B'}}_{2} = \{\varepsilon_{1},\varepsilon_{2},\varepsilon_{3}, 
\varepsilon_{4}\}$, $\underline{\mathcal{B'}}_{3} = \{\varepsilon_{1},\varepsilon_{2},\varepsilon_{3},\varepsilon_{4}\}$, $\underline{\mathcal{B'}}_{4} = \{\varepsilon_{2},\varepsilon_{3},\varepsilon_{4}\}$. If we take $\alpha=\varepsilon_{2}$ and $i=3$, 
then $s=\deg(\alpha+\varepsilon_{3})=41$, $\tau(s)=\varepsilon_{1}+\varepsilon_{4}$ and $f(\varepsilon_{2},3)=x_{2}x_{3}-x_{1}x_{4}=f(\varepsilon_{3},2)$.
\end{example}
\medskip

\begin{lemma}\label{mg5}
Let $e\geq 5$, then 
$$\mathcal{G}=\bigcup_{i=1}^{e-1}\{ f(\epsilon_{j},i)\mid \epsilon_{j}\in {\underline{\mathcal{B'}}_{i}},\quad i\leq j\leq e-1\}$$ is a minimal generating set for $\mathfrak{p}(M)$. 
\end{lemma}

\proof Let us write $n_{i}=a+id$, for $0\leq i \leq e-3,$ $n_{e-2}=b$ and $n_{e-1}=b+d$. 
We consider two main cases: $b\equiv (a-3)i(\mathrm{mod}\, a)$ and 
$b\equiv (a-2)i\, (\mathrm{mod}\, a)$, where $d\equiv i(\mathrm{mod}\, a)$. 
Let $\mathcal{B}=\text{Ap}(\Gamma_{e}(M),a)$.
\medskip

\noindent\textbf{Case 1.} First we assume that $b\equiv (a-3)i(\mathrm{mod}\, a)$, where $d\equiv i(\mathrm{mod}\, a)$. Then, \begin{alignat*}{3}
&\mathcal{B}&&= \{0, a+d,\ldots, a+(e-3)d, b, b+d, b+a+2d\}\\
&\underline{\mathcal{B}} && = \{0,\epsilon_{1},\epsilon_{2},\ldots,\epsilon_{e-3},\epsilon_{e-2},\epsilon_{e-1}, \epsilon_{e-1}+\epsilon_{1}\}\\
&\underline{\mathcal{B}}_{1} && =  \{\epsilon_{1},\epsilon_{2},\ldots,\epsilon_{e-3},\epsilon_{e-2}, \epsilon_{e-1}+\epsilon_{1}\}\\
&\underline{\mathcal{B}}_{2} && =  \{\epsilon_{1},\epsilon_{2},\ldots,\epsilon_{e-2},\epsilon_{e-1}, \epsilon_{e-1}+\epsilon_{1}\}\\
&\,  && \, \,  \vdots\\
&\underline{\mathcal{B}}_{e-2} && =\{\epsilon_{1},\epsilon_{2},\ldots,\epsilon_{e-2} ,\epsilon_{e-1}, \epsilon_{e-1}+\epsilon_{1}\}\\ 
&\underline{\mathcal{B}}_{e-1} && =\{\epsilon_{2},\ldots,\epsilon_{e-3},\epsilon_{e-2} ,\epsilon_{e-1}, \epsilon_{e-2}+\epsilon_{2}\}\\
 &\underline{\mathcal{B'}}_{1} && =  \{\epsilon_{1},\epsilon_{2},\ldots ,\epsilon_{e-2} \}\\
&\underline{\mathcal{B'}}_{2} && =  \{\epsilon_{1},\epsilon_{2},\ldots,\epsilon_{e-2},\epsilon_{e-1}\}\\
&\,  && \, \,  \vdots\\
&\underline{\mathcal{B'}}_{e-2} && =\{\epsilon_{1},\epsilon_{3},\ldots,\epsilon_{e-2} ,\epsilon_{e-1}\}\\ 
&\underline{\mathcal{B'}}_{e-1} && =\{\epsilon_{2},\ldots, \epsilon_{e-3},\epsilon_{e-2} ,\epsilon_{e-1}\}
\end{alignat*}
Now, since $f(\epsilon_{j}, i)= f(\epsilon_{i},j)$, using Theorem 3.7 in 
chapter 3 of \cite{patil} we can say that the set 
$$\mathcal{G}=\bigcup_{i=1}^{e-1}\{ f(\epsilon_{j},i)\mid i\leq j\leq e-1, \epsilon_{j}\in {\underline{\mathcal{B'}}_{i}}\}
\setminus\{f(\epsilon_{e-1},1)\}$$
generates the defining ideal $\mathfrak{p}(M)$. The elements of 
$\mathcal{G}$ are of the following forms:
\begin{enumerate}[(i)]
\item $f_{i}:=f(\epsilon_{i},i)=x_{i}x^{\epsilon_{i}}-x^{\tau(h_{i})}=x_{i}^{2}-x^{\tau(h_{i})}$; where $h_{i}:=\deg(2\epsilon_{i})$ and $1\leq i\leq e-1$.
\item $f_{ij}=f(\epsilon_{j},i)=x_{i}x^{\epsilon_{j}}-x^{\tau(h_{ij})}=x_{i}x_{j}-x^{\tau(h_{ij})}$, where $h_{ij}:=\deg(\epsilon_{i}+\epsilon_{j})$ and $1\leq i<j\leq e-1$.
\end{enumerate}
We claim that $\mathcal{G}$ is also a minimal generating set for the ideal $\mathfrak{p}(M)$.
\medskip

\noindent\textbf{Generators of type (i).} \, Consider
\begin{align*}
f_{i}&=x_{i}^{2}-x^{\tau(h_{i})}= \sum\limits_{l=1,l\neq i}^{e-1}\alpha_{l}(x_{l}^{2}-x^{\tau(h_{l})})+\sum\limits_{l,k, k>l}^{e-1}\alpha_{lk}(x_{l}x_{k}-x^{\tau(h_{lk})})\\
&=\underbrace{\sum\limits_{l=1,l\neq i}^{e-1}\alpha_{l}x_{l}^{2}+\sum\limits_{l,k, k>l}^{e-1}\alpha_{lk}x_{l}x_{k}}_{I}-\underbrace{\left(\sum\limits_{l=1,l\neq i}^{e-1}\alpha_{l}x^{\tau(h_{l})}+\sum\limits_{l,k, k>l}^{e-1}\alpha_{lk}x^{\tau(h_{lk})}\right)}_{II}.
\end{align*}
From the above equation it is clear that $x_{i}^{2}$ can occur only in part $II$ 
of the above equation. Now we consider the following cases.
\medskip

\noindent\textbf{Case (a).} Suppose that $\alpha_{l}=c_{l}x_{i}$ or $\alpha_{lk}=c_{lk}x_{i}$, 
for some $1\leq l<k\leq (e-1)$, where $c_{l}, c_{lk}\in k$. Therefore, either 
$f_{l}=x_{l}^{2}-x_{i}$ or $f_{lk}=x_{l}x_{k}-x_{i}$. 
If $f_{l}=x_{l}^{2}-x_{i}$, then $2n_{l}=n_{i}$, which is a contradiction. 
If $f_{lk}=x_{l}x_{k}-x_{i}$, then $n_{l}+n_{k}=n_{i}$ gives a contradiction 
to the minimality of the generating set $M$.
\medskip
 
\noindent\textbf{Case (b).} Suppose that $\alpha_{l}=c_{l}$ and $\alpha_{lk}=c_{lk}$, 
for some $1\leq l<k\leq (e-1)$, where $c_{l}, c_{lk}\in k$. Therefore, either 
$f_{l}=x_{l}^{2}-x_{i}^{2}$ or $f_{lk}=x_{l}x_{k}-x_{i}^{2}$. Suppose that 
$f_{l}=x_{l}^{2}-x_{i}^{2}$. Since $\tau(2n_{i})\in (\mathbb{Z}_{\geq 0})^{e}$ 
is the unique maximal element of degree $2n_{i}$, we have $x^{\tau (2{n}_{l})}=x^{\tau (2n_{i})}= x_{i}^{2}$, 
which gives $f_{i}=0$ and that is impossible. If $f_{lk}=x_{l}x_{k}-x_{i}^{2}$, 
then $x^{\tau (n_{l}+n_{k})}=x^{\tau (2n_{i})}= x_{i}^{2}$. This gives us $f_{i}=0$, 
which contradicts part (iv) of Lemma \ref{patillemma}. Therefore the set $\{f_{i}\mid 1\leq i\leq (e-1)\}$ 
is a subset of a minimal generating set.
\medskip

\noindent\textbf{Generators of type (ii).} \, Let $1\leq i, j \leq e-1$. Consider 
\begin{align*}
f_{ij}&=x_{i}x_{j}-x^{\tau(h_{ij})}\\
&=\sum\limits_{l=1}^{e-1}\alpha_{l}(x_{l}^{2}-x^{\tau(h_{l})})+\sum\limits_{l,k, k>l,l\neq i, }^{e-1}\alpha_{lk}(x_{l}x_{k}-x^{\tau(h_{lk})})\\
&+ \sum\limits_{ k>i,k\neq j}^{e-1}\alpha_{ik}(x_{i}x_{k}-x^{\tau(h_{ik})})\\
&=\underbrace{\sum\limits_{l=1}^{e-1}\alpha_{l}x_{l}^{2}+\sum\limits_{l,k, k>l,l\neq i}^{e-1}\alpha_{lk}x_{l}x_{k}}_{I}\\
&-\underbrace{\left(\sum\limits_{l=1}^{e-1}\alpha_{l}x^{\tau(h_{l})}+\sum\limits_{l,k, k>l,l\neq i}^{e-1}\alpha_{lk}x^{\tau(h_{lk})}+\sum\limits_{ k>i,k\neq j}^{e-1}\alpha_{ik}x^{\tau(h_{ik})}\right)}_{II}.
\end{align*}
We note that $x_{i}x_{j}$ can occur only in the part $II$ of the above equation. 
Then following cases will occur.
\medskip

\noindent\textbf{Case (a).} Suppose that $\alpha_{l}=c_{l}x_{i}$ or $\alpha_{lk}=c_{lk}x_{i}$, 
where $c_{l}, c_{lk}\in k$, for some $1\leq l<k\leq (e-1)$. Then, we must have: 
$f_{l}=x_{l}^{2}-x_{j}$, $f_{lk}=x_{l}x_{k}-x_{j}$ or $f_{ik}=x_{i}x_{k}-x_{j}$. 
Each case gives a contradiction to the fact that $M$ is a minimal generating set of 
$\Gamma_{e}(M)$. 
\medskip

\noindent\textbf{Case (b).} Suppose that $\alpha_{l}=c_{l}x_{j}$ or $\alpha_{lk}=c_{lk}x_{j},$ 
where $c_{l}, c_{lk}\in k$, for  $1\leq l<k\leq (e-1)$. A similar argument works for 
this case. Now, let us assume that $\alpha_{l}=c_{l}$ and $\alpha_{lk}=c_{lk}$, 
where $c_{l}, c_{lk}\in k$, $1\leq l<k\leq (e-1)$. Then, one of these must hold: 
$f_{l}=x_{l}^{2}-x_{i}x_{j}$, $f_{lk}=x_{l}x_{k}-x_{i}x_{j}$, $f_{ik}= x_{l}x_{k}-x_{i}x_{j}$. 
If $f_{l}=x_{l}^{2}-x_{i}x_{j}$, then since $\tau(2n_{l})\in (\mathbb{Z}_{\geq 0})^{e}$ is 
the unique maximal element of degree $2n_{l}$, therefore $x^{\tau (2{n_{l}})}=x^{\tau (n_{i}+n_{j})}= x_{i}x_{j}$. 
This gives $f_{j}=0$, which is a contradiction to Lemma \ref{patillemma}. If $f_{lk}=x_{l}x_{k}-x_{i}x_{j}$, then $\tau(n_{l}+n_{k})\in (\mathbb{Z}_{\geq 0})^{e}$ being the unique maximal element of degree $(n_{i}+n_{j})$, 
we have $X^{\tau (n_{l}+n_{k})}=X^{\tau (n_{i}+n_{j})}= x_{l}x_{k}$. This gives $f_{ij}=0$, 
which is a contradiction to Lemma \ref{patillemma}. Lastly, if 
$f_{ik}=x_{i}x_{k}-x_{i}x_{j}$, then it leads to a contradiction to 
the minimality of the generating set $M$ of the semigroup $\Gamma_{e}(M)$.
\medskip

Therefore, the set $\{f_{ij}\mid 1\leq i\leq e-1,\, i<j\leq e-1\} $ is a subset of a minimal generating 
set $\mathcal{G}.$ From the cases (i) and (ii), it is clear that $\mathcal{G}$ forms a minimal generating 
set for the ideal $\mathfrak{p}(M)$. It can be easily seen that the cardinality of the set $\mathcal{G}$ 
is $\sum\limits_{i=1}^{e-1}\sum\limits_{j\geq i}^{e-1}(\epsilon_{j},i) = \displaystyle\frac{e(e-1)}{2}-1$.
\medskip

\noindent\textbf{Case 2.} We now assume that $b\equiv (a-2)i\, (\mathrm{mod}\, a)$, where $d\equiv i(\mathrm{mod}\, a)$. 
Then, 
\begin{alignat*}{3}
&\mathcal{B}&&= \{0, a+d,\ldots, a+(a-4)d, b, b+d, (a+(e-p+1)d)+(a+(p-3)d)\}\\
&\underline{\mathcal{B}} && = \{0,\epsilon_{1},\ldots,\epsilon_{e-2},\epsilon_{e-1}, \epsilon_{1}+\epsilon_{e-3}\}\\
&\underline{\mathcal{B}}_{1} && =  \{\epsilon_{1},\ldots ,\epsilon_{e-3},\epsilon_{e-2}, \epsilon_{1}+\epsilon_{e-3}\}\\
&\underline{\mathcal{B}}_{2} && =  \{\epsilon_{1},\ldots ,\epsilon_{e-2},\epsilon_{e-1}, \epsilon_{1}+\epsilon_{e-3}\}\\
&\,  && \, \,  \vdots\\
&\underline{\mathcal{B}}_{e-4} &&  =  \{\epsilon_{1},\epsilon_{2},\ldots ,\epsilon_{e-2},\epsilon_{e-1}, \epsilon_{1}+\epsilon_{e-3}\}\\
&\underline{\mathcal{B}}_{e-3} && =\{\epsilon_{2},\epsilon_{3},\ldots,\epsilon_{e-2}, \epsilon_{e-1}, \epsilon_{1}+\epsilon_{e-3}\}\\
&\underline{\mathcal{B}}_{e-2} &&  =  \{\epsilon_{1},\epsilon_{2},\ldots ,\epsilon_{e-2},\epsilon_{e-1}, \epsilon_{1}+\epsilon_{e-3}\}\\
&\underline{\mathcal{B}}_{e-1} &&  =  \{\epsilon_{1},\epsilon_{2},\ldots ,\epsilon_{e-2},\epsilon_{e-1}, \epsilon_{1}+\epsilon_{e-3}\}\\
&\,  && \, \,  \vdots\\ 
 &\underline{\mathcal{B'}}_{1} && =  \{\epsilon_{1}, \epsilon_{2},\ldots,\epsilon_{e-2} \}\\
&\underline{\mathcal{B'}}_{2} && =  \{\epsilon_{1},\epsilon_{2},\ldots,\epsilon_{e-2},\epsilon_{e-1}\}\\
&\,  && \, \,  \vdots\\
&\underline{\mathcal{B'}}_{e-4} &&  =  \{\epsilon_{1},\epsilon_{2},\ldots ,\epsilon_{e-2},\epsilon_{e-1}\}\\
&\underline{\mathcal{B'}}_{e-3} && =\{\epsilon_{2},\epsilon_{3},\ldots,\epsilon_{e-2}, \epsilon_{e-1}\}\\
&\underline{\mathcal{B'}}_{e-2} &&  =  \{\epsilon_{1},\epsilon_{2},\ldots ,\epsilon_{e-2},\epsilon_{e-1}\}\\
&\underline{\mathcal{B'}}_{e-1} && =\{\epsilon_{1},\epsilon_{2},\ldots, \epsilon_{e-2},\epsilon_{e-1}\}
\end{alignat*}

\noindent By a similar argument as in the previous case, the set 
$$\mathcal{G}=\bigcup_{i=1}^{e-1}\{ f(\epsilon_{j},i)\mid i\leq j\leq e-1, \epsilon_{j}\in {\underline{\mathcal{B'}}_{i}}\}\setminus\{f(\epsilon_{e-1},1)\}$$ 
is a minimal generating set for defining ideal $\mathfrak{p}(M)$. \qed

\begin{example}
Let $\Gamma_{5}$ be as in \ref{example=5}, we use GAP computer algebra system \cite{GAP4} to compute $\tau(s)$. 
It follows that the set 
\begin{align*}
\mathcal{G}&=\{f(\varepsilon_{1},1),f(\varepsilon_{2},1),f(\varepsilon_{3},1),f(\varepsilon_{2},2),f(\varepsilon_{3},2),f(\varepsilon_{4},2),f(\varepsilon_{3},3),f(\varepsilon_{4},3),f(\varepsilon_{4},4) \}\\
&=\{x_{1}^{2}-x_{0}x_{2},x_{1}x_{2}-x_{0}^{2}x_{3},x_{1}x_{3}-x_{0}x_{4},x_{2}^{2}-x_{0}^{2}x_{4}, x_{2}x_{3}-x_{1}x_{4},\\
&\quad x_{2}x_{4}-x_{0}^{8}, x_{3}^{2}-x_{0}^{7}, x_{3}x_{4}-x_{0}^{6}x_{1}, x_{4}^{2}-x_{0}^{6}x_{2} \}
\end{align*}
forms a minimal generating set for the ideal $\mathfrak{p}(M)$.
\end{example}

\bibliographystyle{amsalpha}

\end{document}